\documentclass[oneside,a4paper,english, 11pt]{amsart}

\usepackage{adjustbox}
\usepackage[latin9]{inputenc}
\usepackage{hyperref}
\usepackage{lscape}
\usepackage{mathrsfs}  
\usepackage{enumerate}  
\usepackage{subcaption}
\usepackage{bm}

\makeatletter
 \theoremstyle{plain}
\newtheorem{thm}{Theorem}[section]
\theoremstyle{plain}
  \newtheorem{prop}[thm]{Proposition}
\theoremstyle{plain}
 \newtheorem{lemma}[thm]{Lemma}
\theoremstyle{plain}
 
\theoremstyle{plain}

\theoremstyle{plain}
\newtheorem{cor}[thm]{Corollary}
\theoremstyle{definition}
  
\theoremstyle{definition}
  
 \theoremstyle{definition}

\theoremstyle{remark}
\newtheorem{rmk}[thm]{Remark}
\numberwithin{equation}{section}

\usepackage{amsmath}
\usepackage{amssymb}
\usepackage{amscd}
\usepackage{amsthm}
\usepackage{array,multirow}
\usepackage[arrow,matrix,tips,all,cmtip,poly,color]{xy}

\usepackage{stmaryrd}
\usepackage{url}
\usepackage{xcolor}
\usepackage{caption}
\usepackage{enumitem}
\usepackage{tikz-cd}
\usetikzlibrary{arrows}
\usetikzlibrary{decorations,decorations.text,decorations.markings,decorations.shapes}
\tikzset{ closed/.style = {decoration = {markings, mark = at position 0.5 with { \node[transform shape, xscale = .8, yscale=.4] {/}; } }, postaction = {decorate} },
open/.style = {decoration = {markings, mark = at position .5 with { \node[transform shape, scale =1.2] {$\circ$}; } }, postaction = {decorate} }
}
\usepackage{float}

\pdfpageheight\paperheight
\pdfpagewidth\paperwidth
\newcommand{\Z}{\mathbb{Z}}

\newcommand{\Q}{\mathbb{Q}}

\newcommand{\Qp}{\mathbb{Q}_p}

\newcommand{\R}{\mathbb{R}}

\newcommand{\F}{\mathbb{F}}

\newcommand{\fm}{\mathfrak{m}}

\newcommand{\bA}{\mathbb{A}}

\newcommand{\bT}{\mathbb{T}}

\newcommand{\cO}{\mathcal{O}}

\newcommand{\cR}{\mathcal{R}}

\newcommand{\eps}{\varepsilon}

\newcommand{\Gal}{\mathrm{Gal}}
\newcommand{\Hom}{\mathrm{Hom}}

\newcommand{\Res}{\mathrm{Res}}
\newcommand{\Ind}{\mathrm{Ind}}

\newcommand{\GL}{\mathrm{GL}}

\newcommand{\Adm}{\mathrm{Adm}}

\newcommand{\semis}{\mathrm{ss}}

\newcommand{\Fp}{\F_p}

\newcommand{\tld}[1]{\widetilde{#1}}

\newcommand{\JH}{\mathrm{JH}}

\newcommand{\rbar}{\overline{r}}
\newcommand{\rhobar}{\overline{\rho}}

\newcommand{\defeq}{\stackrel{\textrm{\tiny{def}}}{=}}

\newcommand{\ovl}[1]{\overline{#1}}

\newif\iffinalrun
\iffinalrun
  \newcommand{\mar}[1]{}
\else
  \newcommand{\mar}[1]{\marginpar{\raggedright\tiny #1}}

\DeclareMathOperator{\Conv}{Conv}

\newcommand{\ra}{\rightarrow}

\newcommand{\onto}{\twoheadrightarrow}

\title{Generic decompositions of Deligne--Lusztig representations}

\author{Daniel Le}
\address{Department of Mathematics,
Purdue University,
150 N. University Street, 
West Lafayette, IN 47907-2067}
\email{ledt@purdue.edu}

\author{Bao V.~Le Hung}
\address{Department of Mathematics,
Northwestern University, 
2033 Sheridan Road\\
Evanston, IL 60208, USA}
\email{lhvietbao@googlemail.com}

\author{Brandon Levin}
\address{Department of Mathematics,
Rice University, 
6100 Main Street,
Houston, Texas 77005, USA}
\email{bl70@rice.edu}

\author{Stefano Morra}
\address{LAGA, UMR 7539, CNRS, Universit\'e Paris 13 - Sorbonne Paris Cit\'e, 
Universit\'e de Paris 8,
99 avenue Jean Baptiste Cl\'ement,
93430 Villetaneuse,
France }
\email{morra@math.univ-paris13.fr}

\begin{document}

\begin{abstract}
Let $G_0$ be a reductive group over $\F_p$ with simply connected derived subgroup, connected center and Coxeter number $h+1$. 
We extend Jantzen's generic decomposition pattern from $(2h-1)$-generic to $h$-generic Deligne--Lusztig representations, which is optimal.
We also prove several results on the ``obvious'' Jordan--H\"older factors of general Deligne--Lusztig representations.
As an application we improve the weight elimination result of \cite{LLL}.
\end{abstract} 

\maketitle

\section{Introduction}
\label{sec:intro}
Let $G_0$ be a reductive group over $\F_p$ with simply connected derived subgroup. An important problem is to understand the representations of the finite group of Lie type $G_0(\F_p)$.  
In characteristic 0, Deligne--Lusztig gave a beautiful geometric construction of characters of $G_0(\F_p)$ which effectively describes all irreducibles.  
It is natural to ask how characteristic 0 irreducibles decompose modulo a prime $\ell$.  When $\ell \neq p$, this problem has been studied extensively see e.g.~\cite{bonnafe-rouquier,bonnafe-dat-rouquier}.  
In contrast, the defining characteristic $\ell = p$ case seems underdeveloped, despite its connections to number theory and more specifically the study of congruences of automorphic forms. 
The main result here, due to Jantzen \cite{jantzen}, describes the mod $p$ reduction of \emph{sufficiently generic} Deligne--Lusztig representations.

Let $T$ be a maximally split $\Fp$-rational maximal torus of $G\defeq G_0\otimes_{\Fp}\ovl{\F}_p$, $\mu$ a character of $T$ and $s$ an element in the Weyl group of $W$ (with respect to $T$). 
To this data we can associate an $\Fp$-rational maximal torus $T_s$ and a $W(\ovl{\F}_p)^\times$-valued character $\theta(s,\mu)$ of $T_s(\Fp)$ and thus a Deligne--Lusztig representation $R_s(\mu)$ (see \S \ref{sec:ing}). 
Recall that in characteristic $p$,  after choosing a Borel subgroup, $p$-restricted highest weights $\lambda$ parametrize irreducible representations $F(\lambda)$ of $G_0(\F_p)$.
We refer to these as Serre weights.
Let $h+1$ denote the Coxeter number of $G_0$ and $\eta$ be (a lift of) the half of the sum of the positive roots.
When $\mu-\eta$ is $(2h-1)$-deep in the base $p$-alcove $C_0$ (anchored at $-\eta$), Jantzen \cite{jantzen} (and Gee, Herzig and Savitt \cite{herzig-duke,GHS} for reductive groups) gives a formula for the reduction $\ovl{R}_s(\mu)$ of $R_s(\mu)$ in terms of Frobenius kernel multiplicities.

A basic feature of Jantzen's formula is that the highest weights of the Jordan--H\"older factors is given by universal combinatorial formulas in $s,\mu$, which in particular implies that the length of the reduction is independent of $s$ and $\mu$.
In order for these universal formulas (as $s$ varies in $W$) to produce $p$-restricted weights, $\mu-\eta$ must be $h$-deep (see Remark \ref{rmk:sharp}). 
Thus the best one can hope for is that Jantzen's formula always holds when $\mu-\eta$ is $h$-deep.
Our first main result confirms this when the center of $G$ is connected.

\begin{thm}\label{thm:main:intro}
Suppose that $G$ has connected center, $\mu-\eta$ is $h$-deep in the base $p$-alcove, and $\lambda$ is a $p$-restricted dominant weight. 

Then $[\ovl{R}_s(\mu):F(\lambda)] \neq 0$ if and only if there exist $\tld{w}$ and $\tld{w}_\lambda$ in the extended affine Weyl group $\tld{W}$ such that:
\begin{itemize}
\item $\tld{w}\cdot C_0$ is dominant and $\tld{w}_\lambda\cdot C_0$ is $p$-restricted (where $\cdot$ denotes the $p$-dot action);
\item
$\tld{w} \uparrow \tld{w}_h \tld{w}_\lambda$ (where $\uparrow$ is the semi-infinite Bruhat order defined as in \cite[II.6.5]{RAGS}, $\tld{w}_h\defeq w_0t_{-\eta}$ and $w_0$ is the longest element of $W$); and
\item $\lambda = \tld{w}_\lambda \cdot (\mu - \eta + s\pi(\tld{w}^{-1}(0)))$ (where $\pi$ denotes the automorphism of $X^*(T)$ corresponding to Frobenius; see \S \ref{sec:notation}). 
\end{itemize}
Moreover, in this case:
\[
[\ovl{R}_s(\mu):F(\lambda)]_{G_0(\Fp)} = [\widehat{Z}(1,\mu+(s\pi-p)(\tld{w}^{-1}(0))+(p-1)\eta):\widehat{L}(1,\lambda)]_{G_1T}. 
\]
Here, $\widehat{Z}(1,-)$ and $\widehat{L}(1,-)$ are the baby Verma/standard and simple modules, respectively, for the augmented Frobenius kernel $G_1T$ (see \S \ref{sec:ing}).  
\end{thm}

\begin{rmk}
\begin{enumerate}
\item
The hypothesis on connected center is necessary for the second conclusion to hold---without it the right hand side  should be replaced by a sum of Frobenius kernel multiplicities.
The (generic) decomposition problem in the general reductive case can often be reduced to Theorem \ref{thm:main:intro} by an analysis of isogenies.

\item 
The multiplicity $[\widehat{Z}(1,\mu'):\widehat{L}(1,\lambda)]_{G_1T}$ can be nonzero only if $\mu',\lambda$ are in the same $p$-dot orbit of the affine Weyl group, in which case, for $p$ sufficiently large, it depends combinatorially on the $p$-facets containing $\mu'$ and $\lambda$. 
For instance, when $p\gg h$ it is controlled by periodic Kazhdan--Lusztig polynomials.
\item In a different direction, Pillen \cite{PillenDL} analyzes the contribution of the $p$-singular weights when $\mu$ lies in exactly one wall of $C_0$ and is $2h-1$ away from the other walls.
\item 
This result was first suggested by considerations in the theory of local models for potentially crystalline Emerton--Gee stacks.
Specifically, the hypothesis on $\mu$ in Theorem \ref{thm:main:intro} is the range where the special fiber cycles (which are expected to reflect the mod $p$ reduction of Deligne--Lusztig representations by the Breuil--M\'ezard conjecture) have uniform behavior.
\end{enumerate}
\end{rmk}
It is also natural to contemplate the dual question, i.e.~given a Serre weight $F(\lambda)$, for which $(s, \mu)$ is $F(\lambda)$ a Jordan--H\"older factor of $\ovl{R}_s(\mu)$? 
This problem is essentially equivalent to decomposing the characteristic zero lift of a  projective cover $P_{\lambda}$ of $F(\lambda)$ into irreducibles (the bulk of which are of the form $R_s(\mu)$).
When $\lambda$ is $2h$-deep in its alcove, the complete decomposition can be obtained from Theorem \ref{thm:main:intro}.
In particular there are always $|W|$ ``obvious'' $R_s(\mu)$ which contain $F(\lambda)$ with multiplicity one.
However, when $\lambda$ is not $2h$-deep the decomposition of $P_\lambda$ becomes considerably more complicated; for instance some $R_s(\mu)$ factors that appear generically may disappear.

Nevertheless, we show that the ``obvious'' $R_s(\mu)$ factors of $P_\lambda$ persist up to essentially the optimal threshold:
\begin{thm}
\label{thm:2:intro}
Suppose that $G$ has connected center.
Let $\lambda$ be a $p$-restricted dominant weight which is $h$-deep in its $p$-alcove.
Then for all $s\in W$, $F(\lambda)$ is a Jordan--H\"older factor of $\ovl{R}_s(\tld{w}_h\cdot \lambda+\eta)$ with multiplicity one. 
\end{thm}

\begin{rmk}
In fact we prove the theorem under weaker hypotheses on $\lambda$ depending on its alcove, see Theorem \ref{thm:JHfactor}.
\end{rmk}

Theorem \ref{thm:2:intro} gives a large supply of characteristic zero irreducible representations containing $F(\lambda)$ when $\lambda$ is $h$-deep. For the number theoretic application discussed below, we would like to construct for \emph{any} ($p$-dot regular) $\lambda$ an $R_s(\mu)$ containing $F(\lambda)$ such that $\mu$ is in the base alcove with essentially the same depth as $\lambda$. %
We establish such a statement in Theorem \ref{thm:JHfactor} under a mild ``smallness'' hypothesis which can always be arranged in type $A$.
Note that this is rather subtle because when $\lambda$ is not $h$-deep the most obvious Deligne--Lusztig induction $R_1(\lambda)$ containing $F(\lambda)$ usually fails the depth requirement (because of a small translation when expressing $R_1(\lambda)$ as $R_s(\mu)$ with $\mu-\eta\in C_0$).  
\subsection{A number theoretic application} \label{intro:app}
We now explain how the above results allow us to improve the main theorem of \cite{LLL}.
Recall the global setting of \emph{loc.~cit}.
Let $F/F^+$ be a totally imaginary extension of a totally real field $F^+\neq\Q$ such that $p$ is inert in $F^+$ and splits in $F$. 
Given a reductive group $G_{/F^+}$ which is an outer form for $\GL_n$ and splits over $F$, and such that $G(F^+\otimes_{\Q}\R)$ is compact, and given a compact open subgroup of the form $U=U^pG(\cO_{F^+_p})\leq G(\bA_{F^+}^{\infty})$ and a $G(\cO_{F^+_p})$-module $M$, we define a space 
\[
S(U,M)\defeq \{f: G(F^+)\backslash G(\bA_{F^+}^{\infty})/U\ra M\mid f(gu)=u_p^{-1}f(g)\,\forall g\in G(\bA_{F^+}^{\infty}),\, u\in U\}
\]
of algebraic modular forms. 
It is endowed with a faithful action of a Hecke algebra $\bT$ (with generators indexed by an infinite set of ``good primes'' for $U$, cf.~\cite[\S 4.2.2]{LLL}) for which each maximal ideal $\fm\subseteq \bT$ has an associated continuous semisimple representation $\rbar_{\fm}: G_F\ra\GL_n(\ovl{\F}_p)$ (cf.~\cite[\S 3.4]{CHT}).
We further assume that $\rbar_{\fm}$ is absolutely irreducible.
In \cite{herzig-duke} (later generalized in \cite{GHS}), Herzig made a remarkable conjecture predicting that the set $W(\rbar_{\fm})$ of $p$-\emph{regular} Serre weights $V$ such that $S(U,V)_{\fm}\neq 0$ is given by a combinatorially defined set $W^{?}(\rbar_{\fm}|_{G_{F^+_p}})$ when $\rbar_{\fm}|_{G_{F^+_p}}$ is semisimple. 
We remark that $W^{?}(\rbar_{\fm}|_{G_{F^+_p}})$ is given in terms of the Jordan--H\"older factors of a Deligne--Lusztig representation associated to $\rbar_{\fm}|_{G_{F^+_p}}$.

The weight elimination statement which we obtain is the following:
\begin{thm}
\label{thm:WE:intro}
Assume that $\rbar_{\fm}$ is absolutely irreducible, and that $\rbar_{\fm}|_{G_{F^+_p}}$ is $(2n+1)$-generic.
Then $ W(\rbar_{\fm})\subseteq W^{?}((\rbar_{\fm}|_{G_{F^+_p}})^{\semis})$.
\end{thm}

This result was proven in \cite{LLL} with the assumption that $\rbar_{\fm}|_{G_{F^+_p}}$ is $(6n -2)$-generic instead of $(2n+1)$-generic.  
As in \emph{loc.~cit}.~the main mechanism to show $F(\lambda)\notin W(\rbar_{\fm})$ is to find sufficiently many $\ovl{R}_s(\mu)$ containing it and use $p$-adic Hodge theory constraints implied by the condition $S(U,R_s(\mu))_{\fm}\neq 0$.
In turn, these constraints translate to combinatorial admissibility conditions which exactly match Jantzen's generic pattern for $W^{?}((\rbar_{\fm}|_{G_{F^+_p}})^{\semis})$.
Our representation theoretic results show that we can find all the necessary Deligne--Lusztig representations under weaker genericity hypotheses.

\textbf{Strategy:} Jantzen gives a very general character formula which describes the multiplicity of $R_s(\mu)$ in a certain projective $G_0(\F_p)$-module $Q_{\lambda}$ containing $F(\lambda)$ in terms of Frobenius kernels multiplicities.
As long as $\mu$ is $h$-generic in the lowest $p$-alcove, those multiplicities are controlled by the principal block and hence are independent of $\mu$. 
Under Jantzen's stronger assumption that $\mu$ is $(2h-1)$-generic, any $F(\lambda)$ that can contribute to $R_s(\mu)$ has the property that $Q_\lambda=P_\lambda$ is indecomposable, thus one gets the formula for the multiplicity $[R_s(\mu):F(\lambda)]$ of $F(\lambda)$ in $R_s(\mu)$.
In contrast, the key difficulty when $\mu$ is not $2h$-generic is that $Q_\lambda$ can be decomposable and Jantzen's character formula only gives a formula for certain weighted sums $\sum[R_s(\mu):F(\lambda')][Q_\lambda:P_{\lambda'}]$ over ``packets'' of Serre weights.
Our key observation is that if $\mu$ is $h$-generic and $R_s(\mu)$ occurs in $Q_\lambda$ then $R_s(\mu)$ does not occur in $Q_{\lambda'}$ for any other $\lambda'$ in the packet, and hence the above sum collapses.

The complication arising from packets also occurs in Theorem \ref{thm:2:intro}, and we resolve it in the same way.
It is clear from Jantzen's general formula that $\Hom_{G_0(\Fp)}(Q_{\lambda}, \ovl{R}_s(\tld{w}_h\cdot\lambda+\eta)) \neq 0$.  By a series of delicate estimates in alcove geometry, we show that if $\lambda$ is $h$-deep in its $p$-alcove then $\Hom_{G_0(\Fp)}(Q_{\lambda'}, \ovl{R}_s(\tld{w}_h\cdot\lambda+\eta)) = 0$ for any other $\lambda'$ in the packet.

\subsection{Acknowledgments}
The authors thank the Max Planck Institute and the Hausdorff Center of Mathematics for excellent working conditions.
D.L.~was supported by the National Science Foundation under agreement DMS-2302623 and a start-up grant from Purdue University.
B.LH.~acknowledges
support from the National Science Foundation under grants Nos.~DMS-1952678 and DMS-2302619 and the Alfred P.~Sloan Foundation.
B.L.~was supported by National Science Foundation grants Nos.~DMS-2306369 and DMS-2237237 and the Alfred P.~Sloan Foundation.
S.M.~was supported by the Institut Universitaire de France.

\subsection{Notation}\label{sec:notation}

Let $p$ be a prime. 
Let $G_0$ be a reductive group over $\F_p$. 
Let $\F/\F_p$ denote a finite extension such that $G \defeq G_0 \otimes_{\F_p} \F$ is split. 
We assume throughout that $G$ has simply connected derived subgroup. 
Let $T \subset B \subset G$ denote a maximal torus and a Borel subgroup. 
Let $G_1 \subset G$ denote the kernel of the relative $\text{($p$-)}$Frobenius isogeny $F$ on $G$. 
Let $G_1T \subset G$ denote the subgroup scheme generated by $G_1$ and $T$. 
Let $\Gamma \defeq G_0(\F_p)$. 

Recall the following standard notations:
\begin{itemize}
\item the character group $X^*(T)$ of $T\times_{\F}{\ovl{\F}_p}$;
\begin{itemize}
\item $R\subset X^*(T)$ the set of roots of $G$ with respect to $T$; 
\item $X^0(T) \subset X^*(T)$ the set of elements $\nu$ with $\langle \nu,\alpha^\vee\rangle = 0$ for all $\alpha \in R$; 
\item the root lattice $\Z R \subset X^*(T)$ generated by $R$; 
\item $R^+ \subset R$ the subset of positive roots with respect to $B$, i.e.~the roots occurring in $\mathrm{Lie}(B)$; note that this is the convention in \cite{jantzen} but opposite to \cite{RAGS};
\item $\Delta \subset R^+$ the subset of simple roots; 
\item $X(T)^+\subset X^*(T)$ the dominant weights with respect to $R^+$; 
\item the $p$-restricted set $X_1(T)\subset X(T)^+$ of dominant weights $\lambda$ such that $\langle \lambda,\alpha^\vee\rangle \leq p-1$ for all $\alpha \in \Delta$; 
\item the partial order $\leq$ on $X^*(T)$ and $X^*(T)\otimes_{\Z} \R$ defined by $\lambda\geq \mu$ if $\lambda-\mu \in \R_{\geq 0} R^+$; 
\item for $\nu \in X^*(T)$ or $X^*(T)\otimes_{\Z} \R$, let $h_\nu \defeq \max_{\alpha \in R} \langle \nu,\alpha^\vee\rangle$;
\item the automorphism $\pi$ of $X^*(T)$ such that $F=p\pi^{-1}$ on $X^*(T)$;
\item a choice of $\pi$-invariant $\eta \in X^*(T)$ such that $\langle \eta,\alpha^\vee\rangle = 1$ for all $\alpha\in\Delta$;
\end{itemize}
\item the Weyl group $W$ of $(G,T)$ and $w_0\in W$ its longest element;
\begin{itemize}
\item the extended affine Weyl group $\tld{W} \defeq X^*(T) \rtimes W$, which acts on $X^*(T)$ on the left by affine transformations; for $\nu\in X^*(T)$ we write $t_\nu\in \tld{W}$ for the corresponding element;
\item the affine Weyl group $W_a \defeq \Z R \rtimes W \subset \tld{W}$; 
\item for $\kappa \in X^*(T)$, we write $\kappa_+ \in W\kappa$ for the unique dominant element in its $W$-orbit, and $\Conv(\kappa)\subset X^*(T)\otimes_\Z\R$ for the convex hull of $W\kappa$; note that this operation is subadditive;
\end{itemize}
\item the set of alcoves of $X^*(T)\otimes_{\Z}\R$, i.e.~the set of connected components of 
\[
X^*(T)\otimes_{\Z}\R\setminus\bigcup_{n\in\Z,\alpha\in R}\{\lambda\in X^*(T)\otimes_{\Z}\R\mid \langle \lambda,\alpha^\vee\rangle=n\},
\]
which has a (transitive) left action of $\tld{W}$; 
\begin{itemize}
\item the dominant alcoves, i.e.~alcoves $A$ such that $0<\langle \lambda,\alpha^\vee\rangle$ for all $\alpha\in \Delta, \lambda\in A$;
\item the lowest (dominant) alcove $A_0 = \{\lambda \in X^*(T)\otimes_{\Z}\R\mid 0<\langle \lambda,\alpha^\vee\rangle<1 \textrm{ for all } \alpha\in R^+\}$; 
\item $\Omega\subset \tld{W}$ the stabilizer of the base alcove;
\item the restricted alcoves, i.e.~alcoves $A$ such that $0<\langle \lambda,\alpha^\vee\rangle<1$ for all $\alpha\in \Delta, \lambda\in A$;
\item the set $\tld{W}^+\subset \tld{W}$ of elements $\tld{w}$ such that $\tld{w}(A_0)$ is dominant;
\item the set $\tld{W}_1\subset \tld{W}^+$ of elements $\tld{w}$ such that $\tld{w}(A_0)$ is restricted; 
\item $\tld{w}_h = w_0t_{-\eta} \in \tld{W}_1$; note that $\tld{W}_1 = \tld{W}^+ \cap \tld{w}_h\tld{W}^+$; 
\end{itemize}
\item the set of $p$-alcoves of $X^*(T)\otimes_{\Z}\R$, i.e.~the set of connected components of 
\[
X^*(T)\otimes_{\Z}\R\setminus\bigcup_{n\in\Z,\alpha\in R}\{\lambda\in X^*(T)\otimes_{\Z}\R\mid \langle \lambda+\eta,\alpha^\vee\rangle=np\};
\]
\begin{itemize}
\item a left $p$-dot action of $\tld{W}$ on $X^*(T)$ defined by $(t_\nu w)\cdot \lambda\defeq p\nu+w(\lambda+\eta)-\eta$; this induces a $p$-dot action of $\tld{W}$ on the set of $p$-alcoves whose restriction to $W_a$ is simply transitive;
\item the dominant $p$-alcoves, i.e.~alcoves $C$ such that $0<\langle \lambda+\eta,\alpha^\vee\rangle$ for all $\alpha\in \Delta, \lambda\in C$;
\item the lowest (dominant) $p$-alcove $C_0\subset X^*(T)\otimes_{\Z}\R$ characterized by $\lambda\in C_0$ if $0<\langle \lambda+\eta,\alpha^\vee\rangle<p$ for all $\alpha\in R^+$;
\item the $p$-restricted alcoves, i.e.~alcoves $C$ such that $0<\langle \lambda+\eta,\alpha^\vee\rangle<p$ for all $\alpha\in \Delta, \lambda\in C$;
\end{itemize}

\item the Bruhat order $\leq$ on $W_a$ with respect to $A_0$ (i.e.~using the reflections across walls of $A_0$ as a set of Coxeter generators);
\begin{itemize}
\item the $\uparrow$ order on the set of $p$-alcoves defined in \cite[II.6.5]{RAGS};
\item the $\uparrow$ order on $W_a$ induced from the ordering $\uparrow$ on the set of $p$-alcoves (via the bijection $\tld{w}\mapsto\tld{w}\cdot C_0$);
\item the Bruhat order on $\tld{W}=W_a\rtimes\Omega$ defined by $\tld{w}\delta\leq \tld{w}'\delta'$ if and only if $\tld{w}\leq \tld{w}'$ and $\delta = \delta'$ where $\delta,\delta'\in\Omega$ and $\tld{w},\tld{w}'\in W_a$;
\item the $\uparrow$ order on $\tld{W}$ defined by $\tld{w}\delta\uparrow \tld{w}'\delta'$ if and only if $\tld{w}\uparrow \tld{w}'$ and $\delta = \delta'$ where $\delta,\delta'\in\Omega$ and $\tld{w},\tld{w}'\in W_a$;
\end{itemize}

\end{itemize}
We will assume throughout that $h_\eta < p$ so that $C_0$ is nonempty. 

\section{Lemmata}
\label{sec:lemmata}

In this section, we collect several lemmata, mostly of a root-theoretic nature, that will be used in later sections. 

\begin{lemma}\label{lemma:dominant}
Suppose that $\tld{s},\tld{w}\in \tld{W}$ such that 
\begin{itemize}
\item $\tld{s} \in \tld{W}^+$;
\item $\tld{s}\uparrow \tld{w}$; 
\item $\tld{s}(0) = \tld{w}(0)$; and
\item the closure of some Weyl chamber contains both $\tld{s}^{-1}(0)$ and $\tld{w}^{-1}(0)$. 
\end{itemize}
Then $\tld{s}=\tld{w}$. 
\end{lemma}
\begin{proof}
Since $\tld{s}(0) = \tld{w}(0)$, $W\tld{s}^{-1}(0) = W\tld{w}^{-1}(0)$. 
Since the closure of some Weyl chamber contains $\tld{s}^{-1}(0)$ and $\tld{w}^{-1}(0)$, we have $\tld{s}^{-1}(0) = \tld{w}^{-1}(0)$. 
This implies that $W\tld{s} = Wt_{-\tld{s}^{-1}(0)} = Wt_{-\tld{w}^{-1}(0)} = W\tld{w}$. 
Then $\tld{s}\in \tld{W}^+$ and $\tld{s} \uparrow \tld{w}$ imply that $\tld{s} = \tld{w}$. 
\end{proof}

\begin{lemma}\label{lemma:convex}
Let $C \subset X^*(T) \otimes_\Z \R$ be an open chamber for the action of $W$ (in particular its closure $\ovl{C}$ is a fundamental domain for the action of $W$). 
Let $x\in \ovl{C}$ and $\eps \in X^*(T) \otimes_\Z \R$. 
If $x+\eps'$ is the unique element in $\ovl{C} \cap W(x+\eps)$, then $\eps' \in \Conv(\eps)$. 
\end{lemma}
\begin{proof}
Applying an element of $w\in W$ to $C$, $x$, and $\eps$, we can and do assume without loss of generality that $\ovl{C}$ is the dominant chamber containing $X(T)^+$. 
Suppose that $x+\eps' = w(x+\eps)$ for $w\in W$. 
We will induct on the length of $w$ (allowing $\eps$ to vary). 
If $\ell(w) = 0$, then $\eps' = \eps$, and we are done. 
Suppose that $w = w's_\alpha$ for $\alpha \in \Delta$ and $w'\in W$ with $\ell(w') = \ell(w) - 1$. 
Then $\langle x+\eps,\alpha^\vee\rangle \leq 0 \leq \langle x,\alpha^\vee\rangle$ so that $\langle x+r\eps,\alpha^\vee\rangle = 0$ for some $r\in [0,1]$. 
Then $x+\eps' = w(x+\eps) = w'(x+r\eps+(1-r)s_\alpha(\eps))$. 
Note that $r\eps+(1-r)s_\alpha(\eps)\in\Conv(\eps)$. 
By the inductive hypothesis, $\eps' \in \Conv(r\eps+(1-r)s_\alpha(\eps)) \subset \Conv(\eps)$. 
\end{proof}

\begin{lemma}\label{lemma:barycenter}
Let $x \in X^*(T) \otimes_{\Z} \R$ be in the closure of $A_0$. 
Then either $h_x \leq \frac{h_\eta}{h_\eta+1}$ or 
\begin{equation}\label{eqn:wall}
\langle x,\alpha^\vee\rangle \geq \frac{1}{h_\eta+1}
\end{equation}
for some simple root $\alpha \in \Delta$. 
\end{lemma}
\begin{proof}
Suppose that \eqref{eqn:wall} does not hold for all simple roots $\alpha$. 
Let $\beta = \sum_{\alpha \in \Delta} n_\alpha \alpha$ be a positive root. 
Then $\langle x, \beta^\vee\rangle = \sum_{\alpha \in \Delta} n_\alpha \langle x, \alpha^\vee\rangle < \frac{1}{h_\eta+1} \sum_{\alpha \in \Delta} n_\alpha$. 
On the other hand, $h_\eta \geq \langle \eta, \beta^\vee\rangle = \sum_{\alpha \in \Delta} n_\alpha \langle \eta, \alpha^\vee\rangle = \sum_{\alpha \in \Delta} n_\alpha$. 
Thus $h_x \leq \frac{h_\eta}{h_\eta+1}$. 
\end{proof}

\begin{lemma}\label{lemma:facet}
Suppose that $x,\eps \in X^*(T)\otimes_{\Z}\R$ with $h_\eps \leq\frac{1}{h_\eta+1}$. 
If the closure of an alcove $A_1$ contains $x$, then there exists an alcove $A_2$ whose closure contains $x+\eps$ such that $\ovl{A}_1$ and $\ovl{A}_2$ intersect. 
\end{lemma}
\begin{proof}
We immediately reduce to the case of an irreducible root system. 
Applying $\tld{w}(-)$ for $\tld{w} \in \tld{W}_a$ with $\tld{w}(A_1) = A_0$, we can assume without loss of generality that $A_1$ is $A_0$. 
Suppose that $h_x \leq \frac{h_\eta}{h_\eta+1}$. 
Then $h_{x+\eps} \leq h_x + h_\eps \leq 1$ so that $x+\eps$ is in the closure of $W(A_0)$ and we can take $A_2$ in $W(A_0)$. 

If $h_x > \frac{h_\eta}{h_\eta+1}$, then Lemma \ref{lemma:barycenter} implies that \eqref{eqn:wall} holds for some $\alpha \in \Delta$ which we now fix. 
Let $\alpha_0 = \sum_{\beta \in \Delta} c_\beta \beta$ be the highest root (recall that we are assuming that the root system is irreducible). 
Consider the facet $F$ defined as the intersection of the hyperplanes of the form
\[
\{ \lambda \in X^*(T) \otimes_\Z \R \mid \langle \lambda,\beta^\vee\rangle = 0\}
\]
for all $\beta \in \Delta$ with $\beta\neq \alpha$ and the hyperplane 
\[
\{ \lambda \in X^*(T) \otimes_\Z \R \mid \langle \lambda,\alpha_0^\vee\rangle = 1\}. 
\]
Then $F$ is contained in the closure of $A_0$. 
Let $W' \subset W_a$ be the stabilizer of $F$. 
Recall from \cite[Ch.~5, \S 3, Proposition 1]{bourbaki} that the stabilizer of $F$ in $W_a$ is generated by reflections along hyperplanes passing through $F$. 
In particular, a conjugate $W''$ of $W'$ by a translation in $X^*(T) \otimes_{\Z} \Q$ is a subgroup of $W$ corresponding to a root subsystem. 

We claim that 
\[
C \defeq \{ \lambda \in X^*(T) \otimes_\Z \R \mid \langle \lambda,\beta^\vee\rangle > 0 \, \forall \, \alpha \neq \beta \in \Delta,\, \langle \lambda,\alpha_0^\vee\rangle < 1 \}
\]
is a chamber for the action of $W'$. 
By \cite[Ch.~5, \S1, Proposition 5]{bourbaki} it suffices to show that any hyperplane passing through $F$ does not intersect $C$. 
Let $\gamma \in R^+$ such that $\langle y,\gamma^\vee\rangle \in \Z$ for some $y\in F$. 
If $\gamma =  \sum_{\beta\in \Delta} n_\beta \beta$, then $\langle y,\gamma^\vee \rangle = \langle y,n_\alpha \alpha^\vee\rangle = \frac{n_\alpha}{c_\alpha} \langle y,\alpha_0^\vee\rangle = \frac{n_\alpha}{c_\alpha}$ so that $n_\alpha = 0$ or $c_\alpha$. 
If $n_\alpha = 0$, then $\langle y,\gamma^\vee\rangle = 0$ for all $y\in F$ and $\langle x,\gamma^\vee\rangle > 0$ for all $x \in C$. 
If $n_\alpha = c_\alpha$, then $\langle y,\gamma^\vee\rangle = 1$ for all $y\in F$ and $\langle x,\gamma^\vee\rangle < 1$ for all $x\in C$. 
This establishes the claim. 
We conclude from \cite[Ch.~5, \S 3, Theorem 2]{bourbaki} that the closure $\ovl{C}$ is a fundamental domain for the action of $W'$. 

Let $x+\eps'$ denote the unique element of the $W'$-orbit of $x+\eps$ in $\ovl{C}$. 
By Lemma \ref{lemma:convex}, $\eps' \in \Conv(\eps)$ (note that the convex hull of $W(\eps)$ contains that of $W''(\eps)$) so that in particular $h_{\eps'} \leq h_\eps$. 
Then $\langle x+\eps',\alpha^\vee \rangle \geq \langle x,\alpha^\vee\rangle - h_{\eps'} \geq \frac{1}{h_\eta+1} - h_\eps \geq 0$. 
This implies that $x+\eps'$ is in the closure of $A_0$. 
We can then take $A_2$ so be $w'(A_0)$ where $w'\in W'$ and $w'(x+\eps') = x+\eps$. 
\end{proof}

Given $\lambda\in X^*(T)$ we let $W(\lambda)$ be the virtual representation $\sum_i(-1)^iR^i\Ind_{B^-}^{G}\lambda$ where $B^-$ denotes the Borel opposite to $B$. 
If $\lambda$ is dominant then $W(\lambda)$ is the representation $\Ind_{B^-}^{G}\lambda$, and we write $L(\lambda)$ for its (irreducible) socle.
Recall from \S~\ref{sec:intro} that we defined $\tld{w}_h=w_0t_{-\eta}$ where $w_0$ is the longest element of $W$.

\begin{lemma}\label{lemma:packet}
Let $\lambda_0\in C_0$, $\tld{w}_\lambda\in \tld{W}_1$, $\lambda' \in X_1(T)$, and $\nu \in X(T)^+$. 
Set $\lambda\defeq \tld{w}_\lambda\cdot \lambda_0$. 
\begin{enumerate}[label=(\arabic*)]
\item 
\label{lemma:packet:it:1}
If $[L(\lambda') \otimes L(\pi(\nu)):L(\lambda+p\nu)]_G \neq 0$, then $\lambda+p\nu\uparrow \lambda'+\pi(\nu')$ for some $\nu' \in \Conv(\nu)$. 
\item 
\label{lemma:packet:it:2}
If $\lambda+p\nu\uparrow \lambda'+\pi(\nu')$ for some $\nu' \in \Conv(\nu)$, then $h_\nu \leq \max_{v \in \ovl{A}_0} h_{\tld{w}_h\tld{w}_\lambda(v)} \leq h_\eta$. 
\end{enumerate}
\end{lemma}
\begin{proof}
Suppose that $L(\lambda+p\nu) \in \JH(L(\lambda') \otimes L(\pi(\nu))) \subset \JH(W(\lambda') \otimes L(\pi(\nu)))$. 
Then $L(\lambda+p\nu) \in \JH(W(\lambda'+\pi(\nu')))$ for some $\nu' \in \Conv(\nu)$. 
Moreover, we can assume without loss of generality that $\lambda'+\pi(\nu') \in X(T)^+$ by \cite[Lemma 2.2.2]{MLM}. 
By the linkage principle, we conclude that 
\begin{equation}\label{eqn:uparrow}
\lambda+p\nu \uparrow \lambda'+\pi(\nu') 
\end{equation}
for some $\nu' \in \Conv(\nu)$. 

Suppose now that \eqref{eqn:uparrow} holds. 
If we let $\alpha_0$ be a dominant root such that $\langle \nu,\alpha_0^\vee \rangle = h_\nu$, then \eqref{eqn:uparrow} implies that 
\begin{align*}
(p-1)h_\nu &\leq \langle p\nu-\pi(\nu'),\alpha_0^\vee \rangle \\
& \leq \langle \lambda'-\lambda,\alpha_0^\vee \rangle \\
& =  \langle \lambda'-\tld{w}_\lambda\cdot \lambda_0,\alpha_0^\vee \rangle \\
& \leq \langle (p-1)\eta-\tld{w}_\lambda\cdot \lambda_0,\alpha_0^\vee \rangle \\
& = \langle p\eta-(\tld{w}_\lambda\cdot \lambda_0+\eta),\alpha_0^\vee \rangle \\
& = \langle -pw_0\eta+(w_0\tld{w}_\lambda)\cdot \lambda_0+\eta,-w_0\alpha_0^\vee \rangle \\
& = \langle \tld{w}_h\tld{w}_\lambda\cdot \lambda_0+\eta,-w_0\alpha_0^\vee \rangle 
\end{align*} 
where
\begin{itemize}
\item the second inequality uses that $\alpha_0$ is dominant; and 
\item the third inequality uses $\alpha_0$ is positive and $\lambda' \in X_1(T)$.
\end{itemize}
We deduce that 
\begin{align*}
\frac{p-1}{p}h_\nu &\leq  \langle \tld{w}_h\tld{w}_\lambda(\frac{1}{p}(\lambda_0+\eta)),-w_0\alpha_0^\vee \rangle \\
&< \max_{v \in \ovl{A}_0} \langle \tld{w}_h\tld{w}_\lambda(v),-w_0\alpha_0^\vee \rangle\\
&\leq \max_{v \in \ovl{A}_0} h_{\tld{w}_h\tld{w}_\lambda(v)} \\
&\leq h_{-w_0\eta} = h_\eta \\
&\leq p-1
\end{align*} 
where 
\begin{itemize}
\item the strict inequality uses that $\lambda \in C_0$; and 
\item the fourth inequality uses that $-w_0\eta-\tld{w}_h\tld{w}_\lambda(v) \in X(T)^+$ noting that $\tld{w}_h^{-1}(-w_0\eta-\tld{w}_h\tld{w}_\lambda(v)) = \tld{w}_\lambda(v)$ is in the closure of a restricted alcove. 
\end{itemize}
This implies that $h_\nu < p$ so that $h_\nu = \lceil \frac{p-1}{p}h_\nu \rceil \leq \max_{v \in \ovl{A}_0} h_{\tld{w}_h\tld{w}_\lambda(v)} \leq h_\eta$. 
(The last inequality follows from the fact that $\eta-\tld{w}_h\tld{w}_\lambda(v)$ is dominant for all $v\in \ovl{A}_0$.) 
\end{proof}

Given $m\in\Z$ and a $p$-alcove 
\[
C = \{\mu \in X^*(T) \otimes_{\Z} \R\mid n_\alpha p<\langle \mu+\eta,\alpha^\vee\rangle<(n_\alpha+1)p,\alpha\in R^+\}, 
\]
we say that $\lambda\in X^*(T)$ is $m$-deep in the $p$-alcove $C$ if for all $\alpha\in R^+$, $n_\alpha p+m<\langle \lambda+\eta,\alpha^\vee\rangle<(n_\alpha+1)p-m$. 

\begin{rmk}\label{rmk:depth}
For an $m$-deep weight in an alcove to exist, one must have $p\geq (m+1)(h_\eta+1)$. 
Indeed, suppose that $\lambda$ is $m$-deep in $C_0$. 
Then $\lambda-m\eta$ is dominant.
Then $\langle (m+1)\eta,\alpha^\vee\rangle < p-m$ for all $\alpha\in R$ from which we deduce the desired inequality. 
\end{rmk}

\begin{lemma}\label{lemma:apriori}
Let $s\in W$, $\mu - \eta \in C_0$, and $\lambda \in X_1(T)$. 
If $\tld{w} \in \tld{W}$ is such that $\tld{w} \cdot (\mu-\eta+s\pi\tld{w}^{-1}(0))+\eta\in X(T)^+$ and $\tld{w} \cdot (\mu-\eta+s\pi\tld{w}^{-1}(0)) \leq \tld{w}_h \cdot \lambda$, then $h_{\tld{w}(0)} = h_{\tld{w}^{-1}(0)} \leq h_\eta+1$. 
If $\mu-\eta$ or $\lambda$ is $1$-deep in their respective $p$-alcoves, then $h_{\tld{w}(0)} = h_{\tld{w}^{-1}(0)} \leq h_\eta$. 
\end{lemma}
\begin{proof}
Let $\sigma\in W$ be such that $\sigma\tld{w} \in\tld{W}^+$, and let $\tau$ be the image of $\sigma\tld{w}$ in $W$. 
Letting $\alpha_0$ be the highest root so that $h_{\tld{w}(0)} = \langle \sigma\tld{w}(0),\alpha_0^\vee\rangle$, the hypotheses imply that 
\begin{align} 
(p-1)(h_{\tld{w}(0)}-1) &= ph_{\tld{w}(0)} -h_{\tld{w}(0)} - (p-1)\nonumber \\
	& \label{eqn:apriori} \leq p \langle \tld{w}(0),\sigma^{-1}\alpha_0^\vee\rangle + \langle \tld{w}^{-1}(0),\pi^{-1}((\tau s)^{-1}\alpha_0^\vee)\rangle + \langle \mu,\tau^{-1}\alpha_0^\vee\rangle  \\ 
	&=  \langle p\sigma\tld{w}(0)+\tau(\mu+s\pi(\tld{w}^{-1}(0)),\alpha_0^\vee\rangle \nonumber\\
	&= \langle \sigma\tld{w}\cdot (\mu-\eta+s\pi(\tld{w}^{-1}(0)))+\eta,\alpha_0^\vee \rangle \nonumber\\
	&\leq \langle \tld{w}_h \cdot \lambda+\eta,\alpha_0^\vee \rangle \nonumber\\
	& \leq (p-1)h_\eta, \label{eqn:apriori'} 
\end{align}
from which we deduce that $h_{\tld{w}(0)} \leq h_\eta+1$. 
If $\mu-\eta$ (resp.~$\lambda$) is $1$-deep in its $p$-alcove, then the inequality in \eqref{eqn:apriori} (resp.~\eqref{eqn:apriori'}) is strict. 
The result follows. 
\end{proof}

\begin{lemma}\label{lemma:presentation}
Let $m \geq 0$. 
Suppose that $\mu-\eta \in X^*(T)$ is $m$-deep in $C_0$ and $\sigma(\mu)+p\nu-s\pi\nu-\eta$ is $(-m+1)$-deep in $C_0$ for $\sigma,s \in W$ and $\nu\in X^*(T)$. 
Then $t_\nu \sigma \in \Omega$. 
\end{lemma}
\begin{proof}
Let $\mu-\eta, \sigma,s,$ and $\nu$ be as in the statement of the lemma. 
We first claim that $h_\nu \leq 2$. 
For $\alpha\in R^+$, we have that 
\begin{equation}\label{eqn:negdepth}
-m+1 < \langle \sigma(\mu)+(p-s\pi)\nu,\alpha^\vee\rangle < p+m-1. 
\end{equation}
Using that $-p+m < \langle \sigma(\mu),\alpha^\vee\rangle < p-m$ for all $\alpha \in R^+$, we have that $|p \langle \nu,\alpha^\vee\rangle - \langle s\pi\nu,\alpha^\vee\rangle|\leq 2p-2$ for all $\alpha \in R^+$. 
There exists $\alpha \in R^+$ such that $|\langle \nu,\alpha^\vee\rangle| = h_\nu$ so that $(p-1)h_\nu \leq |p \langle \nu,\alpha^\vee\rangle - \langle s\pi\nu,\alpha^\vee\rangle|\leq 2p-2$. 
The claim follows. 

Since $\mu-\eta \in X^*(T)$ is $m$-deep in $C_0$, for each $\alpha \in R^+$ we have 
\[
n_\alpha p+m < \langle \sigma(\mu) + p\nu,\alpha^\vee \rangle < (n_\alpha+1) p-m 
\]
for a unique $n_\alpha \in \Z$. 
On the other hand, \eqref{eqn:negdepth} and that $h_\nu \leq 2$ imply that 
\[
-m-1 < \langle \sigma(\mu)+p\nu,\alpha^\vee\rangle < p+m+1 
\]
for each $\alpha \in R^+$. 
Together, these imply that $n_\alpha = 0$ for all $\alpha \in R^+$ so that $\sigma(\mu) + p\nu -\eta \in C_0$. 
Equivalently, we have $t_\nu \sigma \in \Omega$. 
\end{proof}

\begin{lemma}\label{lemma:LAP}
Suppose that the center $Z$ of $G$ is connected. 
Suppose also that $\lambda_0,\mu_0 \in X^*(T)$ are in the closure of $C_0$ and $\tld{w}_\lambda,\tld{w}_\mu \in \tld{W}$ such that 
\begin{enumerate}[label=(\arabic*)]
\item \label{item:uparrow} $\pi^{-1}(\tld{w}_\lambda)\cdot \lambda_0 \uparrow \pi^{-1}(\tld{w}_\mu)\cdot \mu_0$; and
\item \label{item:modWa} $t_{\lambda_0}\tld{w}_\lambda W_a = t_{\mu_0} \tld{w}_\mu W_a$. 
\end{enumerate}
Then $\lambda_0 = \mu_0$. 
Furthermore, let $F$ be the facet of $C_0$ determined by $\lambda_0$. 
Then $\pi^{-1}(\tld{w}_\lambda)\cdot F \uparrow \pi^{-1}(\tld{w}_\mu)\cdot F$. 
In particular, if $\lambda_0 \in C_0$, then $\tld{w}_\lambda \uparrow \tld{w}_\mu$. 
\end{lemma}
\begin{proof}
\ref{item:uparrow} implies that 
\begin{equation*}\label{eqn:dotequal}
\pi^{-1}(\tld{w}\tld{w}_\lambda)\cdot \lambda_0 = \pi^{-1}(\tld{w}_\mu)\cdot \mu_0 
\end{equation*}
for some $\tld{w} \in W_a$ so that $\mu_0-\lambda_0 \equiv p\pi^{-1}(\tld{w}_\mu^{-1}\tld{w}_\lambda)(0) \pmod {\Z R}$. 
On the other hand, \ref{item:modWa} implies that $\mu_0-\lambda_0 \cong \tld{w}_\mu^{-1}\tld{w}_\lambda(0) \pmod {\Z R}$. 
We conclude that $(1-p\pi^{-1})\tld{w}_\mu^{-1}\tld{w}_\lambda(0) \in \Z R$. 
If $\pi$ has order $f$, then we conclude that 
\[
(1-p^f) \tld{w}_\mu^{-1}\tld{w}_\lambda(0) = (1+p\pi^{-1}+\ldots+p^{f-1}\pi^{-f+1})(1-p\pi^{-1})\tld{w}_\mu^{-1}\tld{w}_\lambda(0) \in \Z R. 
\]
As $Z$ is connected, $X^*(Z)=X^*(T)/\Z R$ is  $(1-p^f)$-torsion-free so that $\tld{w}_\mu^{-1}\tld{w}_\lambda(0) \in \Z R$ or in other words $\tld{w}_\lambda W_a = \tld{w}_\mu W_a$. 
Recall that the closure of $C_0$ is a fundamental domain for the $p$-dot action of $W_a$ on $X^*(T) \otimes_\Z \R$ in a strong sense: $\ovl{C}_0$ intersects any $W_a$-orbit exactly once (see \cite[Ch.~5, \S 3, Theorem 2]{bourbaki}). 
Then the fact that $\lambda_0,\mu_0$ are in $\ovl{C}_0$ implies that $\lambda_0 = \mu_0$. 

\ref{item:uparrow} implies that there is a sequence of affine reflections $\tld{r}_1,\ldots,\tld{r}_m$ such that 
\begin{itemize}
\item for each $1\leq k\leq m$, $\tld{r}_{k-1}\cdots\tld{r}_1 \pi^{-1}(\tld{w}_\lambda)\cdot \lambda_0$ (resp.~$\tld{r}_k\cdots\tld{r}_1 \pi^{-1}(\tld{w}_\lambda)\cdot \lambda_0$) is in the negative (resp.~positive) halfspace defined by $\tld{r}_k$; and
\item $\tld{r}_m\cdots\tld{r}_1 \pi^{-1}(\tld{w}_\lambda)\cdot \lambda_0 = \pi^{-1}(\tld{w}_\mu)\cdot \mu_0$. 
\end{itemize}
These properties hold after replacing $\lambda_0$ and $\mu_0$ with $F$. 
If $\lambda_0,\mu_0 \in C_0$, then $F = C_0$ and $\tld{w}_\lambda \uparrow \tld{w}_\mu$. 
\end{proof}

We say that $\mu \in X^*(T)$ is $p$-regular if $\mu$ is $0$-deep in its $p$-alcove. 
Recall from \cite[II.7.2]{RAGS} the notion of blocks for $G$.
By the linkage principle (\cite[II.2.12(1) and II.6.17]{RAGS}) if $L(\lambda)$ and $L(\mu)$ are in the same block, then $\lambda$ is $p$-regular if and only if $\mu$ is, in which case we say that the block is $p$-regular.
Given a $G$-module $V$, let $V_{\mathrm{reg}}$ be the projection of $V$ to the $p$-regular blocks. 

\begin{lemma}\label{lemma:translation}
Let $\mu \in X(T)^+$ be $0$-deep in its alcove, and suppose that $\nu \in X(T)^+$ such that $\mu+\kappa$ is in the closure of the $p$-alcove containing $\mu$ for all $\kappa \in \Conv(\nu)$. 
Then 
\[
(L(\mu) \otimes L(\nu))_{\mathrm{reg}} \cong \bigoplus_{\substack{\kappa \in \Conv(\nu) \\ \mu+\kappa \textrm{ is $0$-deep}}} L(\mu+\kappa)^{\oplus [L(\nu)|_T: \kappa]_T}
\]
(and each summand that appears on the RHS has highest weight $\mu+\kappa$ in the same $p$-alcove as $\mu$). 
\end{lemma}
\begin{proof}
The proof is as in \cite[Lemma]{humphreys-generic}, except that we project to $p$-regular blocks. 
As the linkage principle ensures that there are no nontrivial $G$-extensions between the constituents of $(L(\mu) \otimes L(\nu))_{\mathrm{reg}}$, it suffices to prove an equality at the level of formal characters. 
We have 
\begin{align}
\mathrm{ch}\, (W(\mu) \otimes L(\nu))_{\mathrm{reg}} &= \sum_{\kappa \in X^*(T)} [L(\nu)|_T: \kappa]_T\, \mathrm{ch}\, W(\mu+\kappa)_{\mathrm{reg}}\nonumber \\
\label{eqn:reg} &= \sum_{\substack{\kappa \in \Conv(\nu) \\ \mu+\kappa \textrm{ is $0$-deep}}} [L(\nu)|_T: \kappa]_T\, \mathrm{ch}\, W(\mu+\kappa), 
\end{align}
where the first equality follows from a formula of Brauer and the second equality follows from the fact that $W(\mu+\kappa)_{\mathrm{reg}} = 0$ if $\mu+\kappa$ is not $p$-regular by the linkage principle and that $[L(\nu)|_T: \kappa]_T \neq 0$ implies that $\kappa \in \Conv(\nu)$. 
By assumption, the highest weight of each $W(\mu+\kappa)$ appearing in \eqref{eqn:reg} is $0$-deep in the same $p$-alcove as $\mu$. 
If $\mu = \tld{w}\cdot \lambda$ for $\tld{w} \in W_a$ and $\lambda\in C_0$ and $\mu+\kappa$ is $0$-deep in the same $p$-alcove as $\mu$, then the linkage principle implies that there are nonnegative integers $a(\tld{w},\tld{w}')$ for each $\tld{w}'\in W_a$ such that 
\begin{equation}\label{eqn:lusztig}
\mathrm{ch}\, W(\mu+\kappa) = \sum_{\tld{w}'\in W_a} a(\tld{w},\tld{w}')\, \mathrm{ch}\, L(\tld{w}'\cdot (\lambda+w^{-1}(\kappa)))
\end{equation}
where $w \in W$ is the image of $\tld{w}$, the integers $a(\tld{w},\tld{w}')$ are independent of $\mu$ and $\kappa$ by the translation principle \cite[II.7.5]{RAGS}, $a(\tld{w},\tld{w}) = 1$, and $a(\tld{w},\tld{w}') \neq 0$ implies that $\tld{w}' \uparrow \tld{w}$. 
Then \eqref{eqn:reg}, \eqref{eqn:lusztig}, and induction using the partial ordering $\uparrow$ yields 
\[
\mathrm{ch}\, (L(\mu) \otimes L(\nu))_{\mathrm{reg}} = \sum_{\substack{\kappa \in \Conv(\nu) \\ \mu+\kappa \textrm{ is $0$-deep}}} [L(\nu)|_T: \kappa]_T\, \mathrm{ch}\, L(\mu+\kappa).
\]
\end{proof}

We have the following immediate corollary of Lemma \ref{lemma:translation}. 

\begin{cor}\label{cor:translation}
Let $\mu \in X(T)^+$ be $0$-deep in its alcove, and suppose that $\nu \in X(T)^+$ such that $\mu+\kappa$ is in closure of the $p$-alcove containing $\mu$ for all $\kappa \in \Conv(\nu)$. If $\lambda \in X(T)^+$ is $0$-deep in its $p$-alcove and $[L(\mu) \otimes L(\nu):L(\lambda)]_G \neq 0$, then $\lambda = \mu+ \kappa$ for some $\kappa \in \Conv(\nu)$ and $\lambda$ and $\mu$ are in the same $p$-alcove. 
\end{cor}

\section{Ingredients}
\label{sec:ing}

In this section, we summarize key results that we will use to investigate reductions of Deligne--Lusztig representations. 
We assume from now on that $p\geq 2h_\eta$.
Given $\lambda\in X_1(T)$ we let $\widehat{Q}_1(\lambda)$ be the $G$-representation constructed in \cite[II.11.11]{RAGS} and $Q_\lambda \defeq \widehat{Q}_1(\lambda)|_\Gamma$. 
Then $Q_\lambda$ is a projective $\F[\Gamma]$-module (\cite[\S~4.5]{Jantzen-FK}, see also \cite[Theorem 10.4]{HumphreysBook}). 
Let $P_\lambda$ denote a $\F[\Gamma]$-projective cover of $F(\lambda)\defeq L(\lambda)|_\Gamma$. 
We first record a result of Chastkofsky and Jantzen (see \cite[Theorem 1]{chastkofsky} and \cite[Corollar 2]{jantzen} and also \cite[Appendix A]{herzig-duke} for the generalization to reductive groups with simply connected derived subgroup). 

\begin{prop}\label{prop:packet}
For $\lambda\in X_1(T)$, we have
\[
Q_\lambda \cong \bigoplus_{\lambda'\in X_1(T)/(p-\pi)X^0(T)} \bigoplus_{\nu \in X(T)^+} P_{\lambda'}^{\oplus[L(\lambda') \otimes L(\pi(\nu)):L(\lambda+p\nu)]_G}
\]
Moreover, $[L(\lambda) \otimes L(\pi(\nu)):L(\lambda+p\nu)]_G$ is $1$ if $\nu = 0$ and is $0$ otherwise. 
\end{prop}

Given $\lambda\in X^*(T)$, let $\widehat{Z}(1,\lambda)$ be the baby Verma $G_1T$-module of highest weight $\lambda$ as defined in \cite[\S 2.5]{jantzen} and write $\widehat{L}(1,\lambda)$ for its irreducible cosocle. 
Following \cite[\S 3.1]{jantzen}, given $(s,\mu)\in W\times X^*(T)$ we let $T_s$ be the $F$-stable maximal torus $g_s T g_s^{-1}$ where $g_s\in G(\ovl{\F}_p)$ is any element such that $g_s^{-1}F(g_s)\in N_G(T)$ is a lift of $s$ and $\theta(s,\mu): \Gamma\cap T_s\ra E^\times$ be a character defined by $\theta(s,\mu)(t)\defeq [\mu(g_s^{-1}t g_s)]$ where $[\cdot]$ denotes the Teichm\"uller lift.
This data gives rise to the (signed) Deligne--Lusztig induction $R_s(\mu)\defeq \eps_G \eps_{T_s}R_{T_s}^{\theta(s,\mu)}$; this is the (occasionally virtual) $\Gamma$-representation denoted $R_s(1,\mu)$ in \cite[\S 3.1]{jantzen}. 
We implicitly assume $R_s(\mu)$ is defined over $E$ and write $\ovl{R}_s(\mu)$ for the semisimplification of the reduction of a $\Gamma$-stable $\cO$-lattice in $R_s(\mu)$. 
The following is a convenient reformulation of \cite[3.2]{jantzen} which describes the decomposition of $Q_\lambda$ into reductions of Deligne--Lusztig representations. 

\begin{thm}\label{thm:jantzen}
Let $\mu\in X^*(T)$, $s\in W$, and $\lambda \in X_1(T)$. 
Then 
\begin{equation*}\label{eqn:packet}
\dim \Hom_\Gamma(Q_\lambda,\ovl{R}_s(\mu)) = \sum_{\nu\in X^*(T)} [\widehat{Z}(1,\mu+(s\pi-p)\nu+(p-1)\eta):\widehat{L}(1,\lambda)]_{G_1T} 
\end{equation*}
where the left hand side is suitably interpreted for virtual representations $R_s(\mu)$. 

Moreover, $[\widehat{Z}(1,\mu+(s\pi-p)\nu+(p-1)\eta):\widehat{L}(1,\lambda)]_{G_1T} \neq 0$ if and only if there exists $\tld{w} \in \tld{W}$ such that $\tld{w}^{-1}(0) = \nu$, $\tld{w} \cdot (\mu-\eta+s\pi(\tld{w}^{-1}(0)))+\eta$ is dominant, and 
\begin{equation*}\label{eqn:frobkernel}
\tld{w} \cdot (\mu-\eta+s\pi(\tld{w}^{-1}(0))) \uparrow \tld{w}_h\cdot \lambda. 
\end{equation*}
\end{thm}
\begin{proof}
\cite[3.2, (3) and (4)]{jantzen} give that 
\begin{align*}
\dim \Hom_\Gamma(Q_\lambda,\ovl{R}_s(\mu)) &= \sum_{\nu \in X^*(T)} [\widehat{Z}(1,\mu+s\pi\nu+(p-1)\eta):\widehat{L}(1,p\nu+\lambda)]_{G_1T} \\
&= \sum_{\nu \in X^*(T)} [\widehat{Z}(1,\mu+(s\pi-p)\nu+(p-1)\eta):\widehat{L}(1,\lambda)]_{G_1T}. 
\end{align*}

By \cite[Lemma 10.1.5]{GHS}, 
\begin{equation*}\label{eqn:frobker}
[\widehat{Z}(1,\mu+(s\pi-p)\nu+(p-1)\eta):\widehat{L}(1,\lambda)]_{G_1T} \neq 0 
\end{equation*}
if and only if $\sigma \cdot (\mu-\eta+(s\pi-p)\nu) \uparrow w_0 \cdot (\lambda - p\eta)$ for all $\sigma \in W$. 
This is equivalent to $\tld{w}\cdot (\mu-\eta+s\pi(\tld{w}^{-1}(0))) \uparrow \tld{w}_h \cdot \lambda$ where $\tld{w} = wt_{-\nu}$ and $w\in W$ is any element such that $w \cdot (\mu-\eta+(s\pi-p)\nu)+\eta$ is dominant. 
\end{proof}

The following is an immediate corollary of Proposition \ref{prop:packet} and Theorem \ref{thm:jantzen}. 

\begin{cor}\label{cor:jantzen}
Suppose that $\mu \in X^*(T)$, $s\in W$, and $\lambda \in X_1(T)$. 
Then 
\begin{align}
\dim \Hom_\Gamma(Q_\lambda,\ovl{R}_s(\mu)) =&\sum_{\substack{\lambda'\in X_1(T)/(p-\pi)X^0(T) \\ \nu\in X(T)^+}} [\ovl{R}_s(\mu):F(\lambda')]_\Gamma[L(\lambda') \otimes L(\pi(\nu)):L(\lambda+p\nu)]_G 
\label{eq:mults}\\
=& \sum_{\nu\in X^*(T)} [\widehat{Z}(1,\mu+(s\pi-p)\nu+(p-1)\eta):\widehat{L}(1,\lambda)]_{G_1T}. 
\nonumber
\end{align}
\end{cor}

\section{Generic decompositions of Deligne--Lusztig representations}

In this section, we prove our main result on the reductions of generic Deligne--Lusztig representations. 
We begin with a corollary of the results from the last section. 

\begin{cor}\label{cor:jantzendeep}
Let $\mu - \eta \in C_0$ and $\lambda \in X_1(T)$. 
Suppose that $\mu-\eta$ or $\lambda$ is $h_\eta$-deep in its $p$-alcove. 
Then $\dim \Hom_\Gamma(Q_\lambda,\ovl{R}_s(\mu))\neq 0$ if and only if there exist $\tld{w}\in \tld{W}^+$ and $\tld{w}_\lambda \in \tld{W}_1$ such that $\tld{w} \uparrow \tld{w}_h \tld{w}_\lambda$ and $\lambda = \tld{w}_\lambda \cdot (\mu - \eta + s\pi(\tld{w}^{-1}(0)))$. 
In fact, $\dim \Hom_\Gamma(Q_\lambda,\ovl{R}_s(\mu))$ equals 
\begin{equation}\label{eqn:Qmult}
\sum_{\substack{\tld{w}\in \tld{W}^+,\tld{w}_\lambda \in \tld{W}_1 \\ \tld{w} \uparrow \tld{w}_h \tld{w}_\lambda \\ \lambda = \tld{w}_\lambda \cdot (\mu - \eta + s\pi(\tld{w}^{-1}(0)))}} [\widehat{Z}(1,\mu+(s\pi-p)(\tld{w}^{-1}(0))+(p-1)\eta):\widehat{L}(1,\lambda)]_{G_1T}, 
\end{equation}
where every term in \eqref{eqn:Qmult} is nonzero and for each $\tld{w}$ that appears in the sum, $\mu-\eta+s\pi(\tld{w}^{-1}(0))$ is in $C_0$. 
If the center $Z$ of $G$ is connected, then there is only one term in \eqref{eqn:Qmult}. 
\end{cor}
\begin{proof}
Let $\mu$ and $\lambda$ be as in the statement of the corollary. 
Suppose that $\nu\in X^*(T)$ so that $[\widehat{Z}(1,\mu+(s\pi-p)\nu+(p-1)\eta):\widehat{L}(1,\lambda)]_{G_1T} \neq 0$. 
We claim that $h_\nu \leq h_\eta$. 
As in the proof of Theorem \ref{thm:jantzen}, $\sigma \cdot (\mu-\eta+(s\pi-p)\nu) \uparrow w_0 \cdot (\lambda - p\eta)$ for all $\sigma \in W$. %
Let $w\in W$ be the unique element such that $\tld{w} \defeq wt_{-\nu} \in \tld{W}^+$. 
Setting $\sigma=w$, Lemma \ref{lemma:apriori} implies that $h_\nu = h_{\tld{w}^{-1}(0)} \leq h_\eta$. 
We claim that $\mu-\eta+s\pi \nu \in C_0$ using that $h_\nu \leq h_\eta$. 
This is clear if $\mu-\eta$ is $h_\eta$-deep in $C_0$. 
If $\lambda$ is $h_\eta$-deep in its $p$-alcove, then so is $\mu-\eta+s\pi \nu$ as it is in the same $\tld{W}$-orbit under the $p$-dot action, in which case $\mu-\eta+s\pi \nu$ and $\mu-\eta$ must lie in the same $p$-alcove which is $C_0$. 

Theorem \ref{thm:jantzen} implies that $[\widehat{Z}(1,\mu+(s\pi-p)\nu+(p-1)\eta):\widehat{L}(1,\lambda)]_{G_1T} \neq 0$ is equivalent to $\tld{w} \cdot (\mu-\eta+s\pi(\tld{w}^{-1}(0))) \uparrow \tld{w}_h\cdot\lambda$ where $\tld{w}$ is defined in terms of $\nu$ as before. 
This is in turn equivalent to the fact that $\tld{w} \uparrow \tld{w}_h\tld{w}_\lambda$ and $\lambda = \tld{w}_\lambda \cdot (\mu-\eta+s\pi(\tld{w}^{-1}(0)))$ for some $\tld{w}_\lambda \in \tld{W}_1$. 

Finally, we show that only one term in \eqref{eqn:Qmult} is nonzero when the center $Z$ of $G$ is connected. 
Suppose that $\tld{w},\tld{w}' \in \tld{W}^+$ and $\tld{w}_\lambda,\tld{w}'_\lambda\in \tld{W}_1$ such that $\tld{w} \uparrow \tld{w}_h\tld{w}_\lambda$, $\tld{w}'\uparrow \tld{w}_h\tld{w}'_\lambda$ and $\tld{w}_\lambda \cdot (\mu-\eta+s\pi(\tld{w}^{-1}(0)))=\tld{w}'_\lambda \cdot (\mu-\eta+s\pi(\tld{w}'^{-1}(0)))$. 
We apply Lemma \ref{lemma:LAP} with $\lambda_0=\mu-\eta+s\pi(\tld{w}^{-1}(0))$, $\mu_0=\mu-\eta+s\pi(\tld{w}'^{-1}(0))$, $\tld{w}_\lambda=\pi(\tld{w}_\lambda)$ and $\tld{w}_\mu=\pi(\tld{w}'_{\lambda})$ (resp. with $\lambda_0=\mu-\eta+s\pi(\tld{w}'^{-1}(0))$, $\mu_0=\mu-\eta+s\pi(\tld{w}^{-1}(0))$, $\tld{w}_\lambda=\pi(\tld{w}'_{\lambda})$ and $\tld{w}_\mu=\pi(\tld{w}_\lambda)$) to obtain $\tld{w}_\lambda\uparrow\tld{w}'_{\lambda}$ (resp.~$\tld{w}'_\lambda\uparrow\tld{w}_{\lambda}$).
Hence $\tld{w}_\lambda = \tld{w}'_\lambda$ and thus $\tld{w}^{-1}(0) = \tld{w}'^{-1}(0)$ from which we deduce that $\tld{w} = \tld{w}'$.
(Note that condition \ref{item:modWa} in Lemma \ref{lemma:LAP} is satisfied since $\tld{w} \uparrow \tld{w}_h\tld{w}_\lambda$, $\tld{w}'\uparrow \tld{w}_h\tld{w}'_\lambda$ imply $\tld{w}(0)+\eta\equiv \tld{w}_\lambda$, $\tld{w}'(0)+\eta\equiv \tld{w}'_\lambda$ modulo $W_a$, so that $\mu-\eta+s\pi(\tld{w}^{-1}(0))+\pi(\tld{w}_\lambda)\equiv \mu\equiv\mu-\eta+s\pi(\tld{w}'^{-1}(0))+\pi(\tld{w}'_\lambda)$ modulo $W_a$.
Similar arguments will be used to check condition \ref{item:modWa} whenever we invoke Lemma \ref{lemma:LAP}.)
\end{proof}

The following is our main result on the reduction of generic Deligne--Lusztig representations. 

\begin{thm}\label{thm:main}
Suppose that the center $Z$ of $G$ is connected, $\mu-\eta$ is $h_\eta$-deep in $C_0$, and $\lambda \in X_1(T)$. 
Then $[\ovl{R}_s(\mu):F(\lambda)]_\Gamma \neq 0$ if and only if there exist $\tld{w}\in \tld{W}^+$ and $\tld{w}_\lambda \in \tld{W}_1$ such that $\tld{w} \uparrow \tld{w}_h \tld{w}_\lambda$ and $\lambda = \tld{w}_\lambda \cdot (\mu - \eta + s\pi(\tld{w}^{-1}(0)))$. 
Moreover, in this case
\[
[\ovl{R}_s(\mu):F(\lambda)]_\Gamma = [\widehat{Z}(1,\mu+(s\pi-p)(\tld{w}^{-1}(0))+(p-1)\eta):\widehat{L}(1,\lambda)]_{G_1T}. 
\]
\end{thm}
\begin{proof}
If $\Hom(Q_\lambda,\ovl{R}_s(\mu)) = 0$, then the result holds by Corollary \ref{cor:jantzendeep}, and so we assume otherwise. 
Suppose now that $\lambda,\lambda' \in X_1(T)$ so that
\begin{enumerate}[label=(\arabic*)]
\item \label{item:packet} $[L(\lambda') \otimes L(\pi(\nu)):L(\lambda+p\nu)]_G \neq 0$ for some $\nu \in X(T)^+$ and 
\item \label{item:contribute} $[\ovl{R}_s(\mu):F(\lambda')]_\Gamma \neq 0$. 
\end{enumerate}
We will show that $\lambda-\lambda'\in (p-\pi)X^0(T)$. 
Then the result follows from Corollaries \ref{cor:jantzen} and \ref{cor:jantzendeep}. 

By Proposition \ref{prop:packet} and \ref{item:contribute}, we have that $\Hom_\Gamma(Q_{\lambda'},\ovl{R}_s(\mu)) \neq 0$. 
Corollary \ref{cor:jantzendeep} implies that there are $\tld{w}' \in \tld{W}^+$ and $\tld{w}_{\lambda'} \in \tld{W}_1$ such that $\tld{w}'\uparrow \tld{w}_h\tld{w}_{\lambda'}$ and $\lambda' = \tld{w}_{\lambda'} \cdot (\mu - \eta + s\pi(\tld{w}'^{-1}(0)))$. 
In particular, $h_{\tld{w}'^{-1}(0)} = h_{\tld{w}'(0)} \leq h_{\tld{w}_h\tld{w}_{\lambda'}(0)} \leq h_\eta$. 
Similarly, there are $\tld{w} \in \tld{W}^+$ and $\tld{w}_\lambda$ such that $\tld{w}\uparrow \tld{w}_h\tld{w}_\lambda$ and $\lambda = \tld{w}_\lambda \cdot (\mu - \eta + s\pi(\tld{w}^{-1}(0)))$ (and $h_{\tld{w}^{-1}(0)} \leq h_\eta$). 
Since $\mu - \eta$ is $h_\eta$-deep in $C_0$, we conclude that 
\begin{equation}\label{eqn:0deep}
\textrm{$\lambda$ and $\lambda'$ are $0$-deep in $\tld{w}_\lambda\cdot C_0$ and $\tld{w}_{\lambda'}\cdot C_0$, respectively.}
\end{equation}

By \ref{item:packet} and Lemma \ref{lemma:packet}\ref{lemma:packet:it:1}, we have $\lambda+p\nu \uparrow \lambda'+\pi(\nu')$ for some $\nu \in X(T)^+$ and $\nu' \in \Conv(\nu)$ so that 
\begin{equation}\label{eqn:inequality}
(t_\nu \tld{w}_\lambda)\cdot (\mu - \eta + s\pi(\tld{w}^{-1}(0))) \uparrow \tld{w}_{\lambda'} \cdot (\mu-\eta +s\pi(\tld{w}'^{-1}(0)) + \pi(\nu'))
\end{equation}
for some $\nu \in X(T)^+$ and some (possibly different) $\nu'\in \Conv(\nu)$. 
\eqref{eqn:inequality} implies that $\mu-\eta +s\pi(\tld{w}'^{-1}(0)) + \pi(\nu')$, which is in the same $p$-dot orbit as $\mu-\eta+s\pi(\tld{w}^{-1}(0))$, is $0$-deep in its $p$-alcove. 
Let $w \in W_a$ be the unique element such that $w^{-1} \cdot (\mu-\eta +s\pi(\tld{w}'^{-1}(0)) + \pi(\nu'))$ is in $C_0$. 
\eqref{eqn:inequality}  and Lemma \ref{lemma:LAP} then imply that $t_\nu \tld{w}_\lambda \uparrow \tld{w}_{\lambda'} w$ and $w^{-1}\cdot (\mu-\eta +s\pi(\tld{w}'^{-1}(0)) + \pi(\nu')) = \mu - \eta + s\pi(\tld{w}^{-1}(0))$. 
(The fact that condition \ref{item:modWa} in Lemma \ref{lemma:LAP} holds is checked by a similar argument as in the proof of Corollary \ref{cor:jantzendeep}, using moreover $\nu'\equiv \nu$ modulo $W_a$ from $\nu'\in \Conv(\nu)$.)
In particular, we have 
\begin{equation*}\label{eq:boundnu}
\nu + \tld{w}_\lambda(v) \leq \tld{w}_{\lambda'}w(v) 
\end{equation*}
for any $v \in \ovl{A}_0$. 
We assume without loss of generality that $p > h_\eta(h_\eta+1)$ by Remark \ref{rmk:depth}. 
As $\nu' \in \Conv(\nu)$ and $h_\nu \leq h_\eta < \frac{p}{h_\eta+1}$ by Lemma \ref{lemma:packet}\ref{lemma:packet:it:2}, the closures of the alcoves $A_0$ and $w(A_0)$ intersect by Lemma \ref{lemma:facet} (taking $x = \frac{1}{p}(\mu+s\pi(\tld{w}'^{-1}(0)))$ and $\eps = \frac{\pi(\nu')}{p}$, and noting that $x+\eps$ is in the interior of $w\ovl{(A_0)}$, so that $\ovl{A}_2=w\ovl{(A_0)}$ in the notation of Lemma \ref{lemma:facet}), say at $v \in \ovl{A}_0$. 
Thus, $w$ stabilizes $v \in \ovl{A}_0$ (by \cite[Ch.~5, \S 3, Proposition 1]{bourbaki}), and we have
\begin{equation}\label{eq:boundnu'}
\nu + \tld{w}_\lambda(v) \leq \tld{w}_{\lambda'}(v). 
\end{equation}

We now claim that 
\begin{equation}\label{eqn:hbound}
\langle s\pi(\tld{w}'^{-1}(0)) + \pi(\kappa),\alpha^\vee\rangle \leq h_\eta+1
\end{equation} 
for any root $\alpha$ and $\kappa \in \Conv(\nu)$. 
Using that $\Conv(\nu)$ is $W$-invariant and that $\alpha_0$ is a highest root if and only if $\pi^{-1}(\alpha_0)$ is, it suffices to show that $\langle \sigma \tld{w}'^{-1}(0) + \kappa,\alpha_0^\vee\rangle \leq h_\eta+1$ for any $\sigma \in W$, any highest (and thus dominant) root $\alpha_0$, and any $\kappa \in \Conv(\nu)$. 
We in fact claim the following series of inequalities:  
\begin{align*}
\langle \sigma \tld{w}'^{-1}(0) + \kappa,\alpha_0^\vee\rangle &\leq \langle \sigma\tld{w}'^{-1}(0) + \nu,\alpha_0^\vee\rangle\\
&\leq \langle \sigma\tld{w}'^{-1}(0) + \tld{w}_{\lambda'}(v) - \tld{w}_\lambda(v),\alpha_0^\vee\rangle\\
&\leq \langle \eta - \tld{w}_{\lambda'}(0)+\tld{w}_{\lambda'}(v) - \tld{w}_\lambda(v),\alpha_0^\vee\rangle \\
&\leq h_\eta + \langle w_{\lambda'}(v) - \tld{w}_\lambda(v),\alpha_0^\vee\rangle \\ 
&\leq h_\eta + \langle w_{\lambda'}(v),\alpha_0^\vee\rangle \\ 
&\leq h_\eta + \langle v,\alpha_0^\vee\rangle\\
&\leq h_\eta+1,
\end{align*}
where $w_{\lambda'} \in W$ is the image of $\tld{w}_{\lambda'}$ under the projection $\tld{W} \onto W$. 
Indeed, 
\begin{itemize}
\item the first inequality follows from the fact that $\kappa \leq \nu$ for $\kappa \in \Conv(\nu)$ and $\alpha_0$ is dominant; 
\item the second inequality follows from $\nu \leq \tld{w}_{\lambda'}(v) - \tld{w}_\lambda(v)$ by \eqref{eq:boundnu'} and $\alpha_0 \in X(T)^+$;
\item the third inequality follows from $\alpha_0 \in X(T)^+$ and $\sigma \tld{w}'^{-1}(0) \leq \eta - \tld{w}_{\lambda'}(0)$ which uses $\sigma \tld{w}'^{-1}(0) \leq -w_0 \tld{w}'(0)$ (since both sides are in $W(-\tld{w}'(0))$ and the RHS is dominant) and $\tld{w}' \uparrow \tld{w}_h \tld{w}_{\lambda'}$; 
\item the fifth inequality uses $\alpha_0 \in R^+$ and $\tld{w}_\lambda(v) \in X(T)^+$; 
\item the sixth inequality uses $w_{\lambda'}(v) \leq v$ and $\alpha_0 \in X(T)^+$; and 
\item the final inequality uses $v\in \ovl{A}_0$. 
\end{itemize}

We will use Corollary \ref{cor:translation} with $\mu$, $\nu$, and $\lambda$ taken to be $\lambda'$, $\pi\nu$, and $\lambda+p\nu$, respectively, to show that $\lambda+ p\nu$ and $\lambda'$ are in the same $p$-alcove. 
It suffices to check that the hypotheses apply. 
\eqref{eqn:0deep} gives that (the dominant) $\lambda'$ and $\lambda+p\nu$ are $0$-deep in their $p$-alcoves. 
Using that $\mu-\eta$ is $h_\eta$-deep in $C_0$, \eqref{eqn:hbound} implies that $\lambda'+\pi(\kappa) = \tld{w}_{\lambda'}\cdot (\mu-\eta+s\pi(\tld{w}'^{-1}(0))+w_{\lambda'}^{-1}\pi(\kappa))$ is $(-1)$-deep in $\tld{w}_{\lambda'}\cdot C_0$ for any $\kappa \in \Conv(\nu)$, i.e.~that $\lambda'+\pi(\kappa)$ is in the closure of the $p$-alcove containing $\lambda'$ for any $\kappa \in \Conv(\nu)$. 
\ref{item:packet} gives the final hypothesis. 

From the previous paragraph, $\lambda+p\nu$ is $p$-restricted so that $\nu \in X^0(T)$. 
Then $L(\lambda') \otimes L(\pi(\nu)) \cong L(\lambda'+\pi(\nu))$ so that \ref{item:packet} implies that $L(\lambda'+\pi(\nu)) \cong L(\lambda+p\nu)$ which implies that $\lambda-\lambda' = (p-\pi)\nu$. 
\end{proof}

\begin{rmk}
\label{rmk:sharp}
In fact, the bound in Theorem \ref{thm:main} is sharp. 
If $G_0 = \GL_{2/\F_p}$, then $\ovl{R}_s(\mu)$ has $2$ Jordan--H\"older factors if $\mu - \eta$ is $1$-deep, but $\ovl{R}_{(12)}(1,0)$ has $1$ Jordan--H\"older factor. 
\end{rmk}

\section{Deligne--Lusztig reductions containing a simple module}
\label{sec:DL:lift}

In this section, we prove Theorem \ref{thm:2:intro} which exhibits Deligne--Lusztig representations whose reductions contain a fixed simple module (see Theorem \ref{thm:JHfactor}). 

\begin{lemma}\label{lemma:vertex}
Suppose that the center $Z$ of $G$ is connected. 
Let $s\in W$, $\lambda_0 \in C_0$, $\lambda' \in X_1(T)$, $\tld{w}_\lambda \in \tld{W}_1$, $\tld{w}\in \tld{W}$, $\nu\in X(T)^+$, and $\nu' \in \Conv(\nu)$ such that 
\begin{enumerate}[label=(\arabic*)]
\item \label{item:packet0} $t_\nu\tld{w}_\lambda \cdot \lambda_0 \uparrow \lambda'+\pi(\nu')$;
\item \label{item:Weyl} if $\tld{w}_{\lambda'} \in t_\nu\tld{w}_\lambda W_a\cap \tld{W}_1$ such that $\lambda' \in \tld{w}_{\lambda'}\cdot \ovl{C}_0$, then $\lambda'+\pi(\nu') \in \tld{w}_{\lambda'}W\cdot C_0$ and $t_\nu\tld{w}_\lambda \uparrow \tld{w}_{\lambda'}$; 
\item \label{item:C0} $\lambda_0 - s\pi(\tld{w}_h\tld{w}_\lambda)^{-1}(0)+s\pi\tld{w}^{-1}(0) \in C_0$; and 
\item \label{item:genericJH} $\tld{w}\cdot (\lambda_0 - s\pi(\tld{w}_h\tld{w}_\lambda)^{-1}(0)+s\pi\tld{w}^{-1}(0))+\eta \in X(T)^+$ and $\tld{w}\cdot (\lambda_0 - s\pi(\tld{w}_h\tld{w}_\lambda)^{-1}(0)+s\pi\tld{w}^{-1}(0)) \uparrow \tld{w}_h \cdot \lambda'$. 
\end{enumerate}
Then $\nu \in X^0(T)$ and $t_\nu\tld{w}_\lambda \cdot \lambda_0 = \lambda'+\pi(\nu)$. 
\end{lemma}
\begin{proof}
Let $\tld{w}_{\lambda'}$ be as in \ref{item:Weyl}, and let $\lambda'_0 \defeq \tld{w}_{\lambda'}^{-1}\cdot \lambda' \in \ovl{C}_0$. 
Then \ref{item:Weyl} implies that $\lambda'_0+w_{\lambda'}^{-1}\pi(\nu') = \tld{w}_{\lambda'}^{-1}\cdot (\lambda'+\pi(\nu')) \in w \cdot C_0$ for some $w \in W$. 
By \cite[Lemma 2.2.2]{MLM}, $w^{-1} \cdot (\lambda'_0+w_{\lambda'}^{-1}\pi(\nu')) = \lambda'_0+\pi(\nu'')$ for some $\nu''\in\Conv(\nu)$. 
Furthermore, \ref{item:packet0} and Lemma \ref{lemma:LAP} with $\lambda_0$, $\mu_0$, $\tld{w}_\lambda$, and $\tld{w}_\mu$ taken to be $\lambda_0$, $w^{-1} \cdot (\lambda'_0+w_{\lambda'}^{-1}\pi(\nu'))= \lambda'_0+\pi(\nu'')$, $\pi(t_\nu\tld{w}_\lambda)$, and $\pi(\tld{w}_{\lambda'}w)$, respectively imply that $t_\nu \tld{w}_\lambda \uparrow \tld{w}_{\lambda'}w$ and $\lambda_0 = \lambda'_0+\pi(\nu'')$. 

\ref{item:C0}, \ref{item:genericJH}, and Lemma \ref{lemma:LAP} with $\lambda_0$, $\mu_0$, $\tld{w}_\lambda$, and $\tld{w}_\mu$ taken to be $\lambda_0 - s\pi(\tld{w}_h\tld{w}_\lambda)^{-1}(0)+s\pi\tld{w}^{-1}(0)$, $\lambda'_0$, $\pi(\tld{w})$ and $\pi(\tld{w}_h\tld{w}_{\lambda'})$, respectively, imply that $\tld{w}\in \tld{W}^+$, $\tld{w} \uparrow \tld{w}_h\tld{w}_{\lambda'}$, and $\lambda_0 - s\pi(\tld{w}_h\tld{w}_\lambda)^{-1}(0)+s\pi\tld{w}^{-1}(0) = \lambda'_0 = \lambda_0 - \pi(\nu'')$. 
Thus, we have $\nu'' = \pi^{-1}(s)((\tld{w}_h\tld{w}_\lambda)^{-1}(0)-\tld{w}^{-1}(0))$. 
From $\tld{w} \uparrow \tld{w}_h\tld{w}_{\lambda'}\uparrow \tld{w}_h t_\nu \tld{w}_\lambda$ using \ref{item:Weyl}, we also have that $\tld{w}(0)\leq w_0\nu+w_0\tld{w}_\lambda(0)-w_0\eta$. 

Let $w_\lambda\in W$ be the image of $\tld{w}_\lambda$. 
We have the inequalities
\begin{align*}
\nu &\leq \eta-\tld{w}_\lambda(0)+w_0\tld{w}(0) \\
&\leq \eta - \tld{w}_\lambda(0) - w_\lambda \tld{w}^{-1}(0) \\
&= w_\lambda \pi^{-1}(s^{-1}) (\nu'') \\
&\leq \nu
\end{align*}
(the second inequality follows from the fact that $w_0\tld{w}(0)$ is antidominant and $-w_\lambda\tld{w}^{-1}(0) \in Ww_0\tld{w}(0)$). 
Thus, these inequalities are all equalities. 
In particular, $\tld{w}(0) = t_{w_0\nu}\tld{w}_h\tld{w}_\lambda(0)$ and $w_\lambda \tld{w}^{-1}(0) = -w_0\tld{w}(0) \in X(T)^+$. 
The first of these equalities also implies that $w_\lambda(t_{w_0\nu}\tld{w}_h\tld{w}_\lambda)^{-1}(0) = -w_0\tld{w}(0) \in X(T)^+$. 
Lemma \ref{lemma:dominant} with $\tld{s}$ and $\tld{w}$ taken to be $\tld{w}$ and $t_{w_0\nu}\tld{w}_h\tld{w}_\lambda =\tld{w}_ht_\nu \tld{w}_\lambda $, respectively, implies that $\tld{w} = \tld{w}_ht_\nu \tld{w}_\lambda$ (recall that $\tld{w} \uparrow \tld{w}_h\tld{w}_{\lambda'} \uparrow \tld{w}_ht_\nu\tld{w}_\lambda$ by \ref{item:Weyl}). 
In particular, $\tld{w}_h t_\nu \tld{w}_\lambda=\tld{w}\in \tld{W}^+$.
Since also $t_\nu \tld{w}_\lambda\in \tld{W}^+$ (as $\tld{w}_\lambda\in \tld{W}^+$ and $\nu\in X(T)^+$), we deduce that $t_\nu \tld{w}_\lambda \in \tld{W}^+\cap \tld{w}_h^{-1}\tld{W}^+ = \tld{W}_1$. 
Since $\tld{w}_\lambda \in \tld{W}_1$, we must have $\nu \in X^0(T)$. 
This further implies that $\lambda'+\pi(\nu') \in \tld{w}_{\lambda'}\cdot \ovl{C}_0$ so that we can take $w = 1$ above and $t_\nu\tld{w}_\lambda\uparrow \tld{w}_{\lambda'}$. 
We now have inequalities $\tld{w} \uparrow \tld{w}_h \tld{w}_{\lambda'} \uparrow t_{w_0\nu}\tld{w}_h\tld{w}_\lambda = \tld{w}$. 
Thus all these inequalities are equalities and so $t_\nu\tld{w}_\lambda = \tld{w}_{\lambda'}$. 
We conclude that $t_\nu\tld{w}_\lambda \cdot \lambda_0$ and $\lambda'+\pi(\nu')$ are in the same alcove and must be equal by \ref{item:packet0}. 
\end{proof}

In practice, the hypotheses \ref{item:Weyl} and \ref{item:C0} in Lemma \ref{lemma:vertex} are sometimes implied by other hypotheses. 

\begin{lemma}\label{lemma:anydirection}
Suppose that the center $Z$ of $G$ is connected and $p>(h_\eta+1)^2$. 
Let $s\in W$, $\tld{w}\in \tld{W}$, $\tld{w}_\lambda \in \tld{W}_1$, $\lambda_0$ be $\max_{v\in \ovl{A}_0} h_{\tld{w}_h\tld{w}_\lambda(v)}$-deep in $C_0$, $\lambda' \in X_1(T)$, $\nu\in X(T)^+$, and $\nu' \in \Conv(\nu)$ such that 
\begin{itemize}
\item $t_\nu\tld{w}_\lambda \cdot \lambda_0 \uparrow \lambda'+\pi(\nu')$; and
\item $\tld{w}\cdot (\lambda_0 - s\pi(\tld{w}_h\tld{w}_\lambda)^{-1}(0)+s\pi\tld{w}^{-1}(0))+\eta \in X(T)^+$; and 
\item $\tld{w}\cdot (\lambda_0 - s\pi(\tld{w}_h\tld{w}_\lambda)^{-1}(0)+s\pi\tld{w}^{-1}(0)) \uparrow \tld{w}_h \cdot \lambda'$. 
\end{itemize}
Then, letting $\tld{w}_{\lambda'} \in t_\nu\tld{w}_\lambda W_a\cap \tld{W}_1$ such that $\lambda' \in \tld{w}_{\lambda'}\cdot \ovl{C}_0$, we have 
\begin{itemize}
\item $\lambda'$ and $\lambda'+\pi(\nu')$ lie in the same $p$-alcove;
\item
$\lambda_0 - s\pi(\tld{w}_h\tld{w}_\lambda)^{-1}(0)+s\pi\tld{w}^{-1}(0)\in C_0$; and
\item $t_\nu\tld{w}_\lambda\uparrow \tld{w}_{\lambda'}$.
\end{itemize}
\end{lemma}
\begin{proof}
Let $m \defeq \max_{v\in \ovl{A}_0} h_{\tld{w}_h\tld{w}_\lambda(v)}$. 
As $\lambda_0$ is $m$-deep in $C_0$, we have that $\lambda'+\pi(\nu')$ is $m$-deep in its $p$-alcove. 
Lemma \ref{lemma:packet}\ref{lemma:packet:it:2} implies that $h_\nu\leq m$ hence $\lambda'$ and $\lambda'+\pi(\nu')$ lie in the same $p$-alcove.
(The inequality $h_\nu\leq m$, coming from Lemma \ref{lemma:packet}\ref{lemma:packet:it:2},  will be used multiple times in this proof.)
The third bullet point in the statement of the lemma gives 
\begin{equation}\label{eqn:uparrow:lambda'}
\tld{w}\cdot(\lambda_0 - s\pi(\tld{w}_h\tld{w}_\lambda)^{-1}(0)+s\pi\tld{w}^{-1}(0)+\pi(\nu'')) \uparrow \tld{w}_h \cdot (\lambda'+\pi(\nu'))
\end{equation} 
for $\nu'' = \pi^{-1}(w^{-1}w_0)(\nu')\in \Conv(\nu)$ with $w\in W$ the image of $\tld{w}$. 
Moreover, since the LHS of \eqref{eqn:uparrow:lambda'} lies in the same $p$-alcove as $\tld{w}\cdot(\lambda_0 - s\pi(\tld{w}_h\tld{w}_\lambda)^{-1}(0)+s\pi\tld{w}^{-1}(0))$ by the depth bound in the previous paragraph and Lemma \ref{lemma:packet}\ref{lemma:packet:it:2}, $\tld{w}\cdot(\lambda_0 - s\pi(\tld{w}_h\tld{w}_\lambda)^{-1}(0)+s\pi\tld{w}^{-1}(0)+\pi(\nu''))$ is dominant by the second item in the statement of the lemma. 
Moreover it is $m$-deep in its $p$-alcove.
Since $h_{s\pi(\tld{w}_h\tld{w}_\lambda)^{-1}(0)}\leq m$, $\lambda_0 - s\pi(\tld{w}_h\tld{w}_\lambda)^{-1}(0)+s\pi\tld{w}^{-1}(0)+\pi(\nu'')$ and $\lambda_0 +s\pi\tld{w}^{-1}(0)+\pi(\nu'')$ are in the same $p$-alcove, which we denote by $\tld{u}\cdot C_0$ for $\tld{u} \in W_a$.
As $\tld{w}\cdot(\lambda_0 - s\pi(\tld{w}_h\tld{w}_\lambda)^{-1}(0)+s\pi\tld{w}^{-1}(0)+\pi(\nu''))$ is dominant, we conclude that $\tld{w}\tld{u} \in \tld{W}^+$. 
The second and third bullet points and Lemma \ref{lemma:apriori} 
imply that $h_{\tld{w}(0)} \leq h_\eta+1$. 

We claim that $\tld{u}$ fixes a $v_0\in \ovl{A}_0 $. 
Since $h_{\nu''} \leq m$ by Lemma \ref{lemma:packet}\ref{lemma:packet:it:2}, $\lambda_0+\pi(\nu'') \in C_0$. 
Then Lemma \ref{lemma:facet}, taking $x = \frac{1}{p}(\lambda_0+\pi(\nu'')+\eta)$ and $\eps = \frac{1}{p} s\pi\tld{w}^{-1}(0)$ and using $p>(h_\eta+1)^2$, implies that $\tld{u}(\ovl{A}_0)$ and $\ovl{A}_0$ intersect (note that $x+\eps$ above is in the interior of an alcove so that $\ovl{A}_2=\tld{u}(\ovl{A}_0)$ in the notation of Lemma \ref{lemma:facet}). 
We take $v_0 \in \tld{u}(\ovl{A}_0)\cap \ovl{A}_0$.
Then $\tld{u}(v_0)=v_0$ by \cite[Ch.~5, \S 3, Theorem 2]{bourbaki}.

We now summarize our key conclusions. 
\eqref{eqn:uparrow:lambda'} implies that $\tld{w}\tld{u} \uparrow \tld{w}_h\tld{w}_{\lambda'}$ for some $\tld{w}_{\lambda'} \in \tld{W}$ with $\lambda' \in \tld{w}_{\lambda'} \cdot C_0$. 
In particular, $\tld{w}(v_0) \leq \tld{w}_h\tld{w}_{\lambda'}(v_0)$ for some $v_0\in \ovl{A}_0 $. 
Additionally, Lemma \ref{lemma:LAP} with $\lambda_0$, $\mu_0$, $\tld{w}_\lambda$, and $\tld{w}_\mu$ taken to be $\lambda_0$, $\tld{w}_{\lambda'}^{-1}\cdot (\lambda'+\pi(\nu'))$, $\pi(t_\nu\tld{w}_\lambda)$, and $\pi(\tld{w}_{\lambda'})$ implies that $t_\nu \tld{w}_\lambda \uparrow \tld{w}_{\lambda'}$. 

We claim that 
\begin{equation}\label{eqn:h+1}
\langle s\pi\tld{w}^{-1}(0)+\pi(\kappa),\alpha^\vee \rangle \leq h_{\tld{w}_h\tld{w}_\lambda(v_0)}+1
\end{equation} 
for all roots $\alpha$ and $\kappa \in \Conv(\nu)$. 
We will use that for any $\sigma \in W$, 
\begin{align}\label{eqn:v0}
\nonumber \sigma\tld{w}^{-1}(0) &\leq -w_0(\tld{w}(0)_+)\\
\nonumber &= -w_0(\tld{w}(v_0)-w(v_0))_+ \\
\nonumber &\leq -w_0(\tld{w}(v_0)_++ (-v_0)_+) \\
 &= -w_0(\tld{w}(v_0) - w_0v_0) \\
\nonumber &\leq -w_0\tld{w}_h\tld{w}_{\lambda'}(v_0) +v_0 \\
\nonumber &\leq -w_0\tld{w}_h t_\nu \tld{w}_\lambda(v_0)+v_0 \\
\nonumber &= -\nu+\eta-\tld{w}_\lambda(v_0)+v_0. 
\end{align}
Here, \begin{itemize}
\item the first inequality uses that both sides are in the same $W$-orbit and the RHS is dominant; 
\item the third inequality uses that $\tld{w}(v_0) \leq \tld{w}_h\tld{w}_{\lambda'}(v_0)$; and 
\item the fourth inequality uses that $\tld{w}_h\tld{w}_{\lambda'} \uparrow \tld{w}_h t_\nu \tld{w}_\lambda$ and $v_0\in \ovl{A}_0$. 
\end{itemize}
Then for $\kappa \in \Conv(\nu)$, 
\begin{align*}
\langle s\pi(\tld{w}^{-1}(0))+\pi(\kappa),\alpha^\vee \rangle &= \langle \pi^{-1}(s)\tld{w}^{-1}(0)+\kappa,\pi^{-1}(\alpha)^\vee \rangle\\
&\leq \langle (\pi^{-1}(s)\tld{w}^{-1}(0)+\kappa)_+,\alpha_0^\vee \rangle\\
&\leq \langle \pi^{-1}(s)\tld{w}^{-1}(0)_++\kappa_+,\alpha_0^\vee \rangle\\
&\leq \langle-\nu+\eta-\tld{w}_\lambda(v_0)+v_0+\nu,\alpha_0^\vee\rangle \\
&= \langle \eta - \tld{w}_\lambda(v_0) + v_0,\alpha_0^\vee\rangle \\
&\leq h_{\tld{w}_h\tld{w}_\lambda(v_0)}+1
\end{align*}
where $\alpha_0$ is some highest root, the third inequality follows from \eqref{eqn:v0}, and the last inequality follows from the fact that $v_0 \in \ovl{A}_0$. 

As $\lambda_0 - s\pi(\tld{w}_h\tld{w}_\lambda)^{-1}(0)+s\pi\tld{w}^{-1}(0)+\pi(\nu'')$ is $m$-deep in its $p$-alcove, \eqref{eqn:h+1} implies that $\lambda_0 - s\pi(\tld{w}_h\tld{w}_\lambda)^{-1}(0)$ is $(-1)$-deep in the same $p$-alcove. 
On the other hand, we know that $\lambda_0 - s\pi(\tld{w}_h\tld{w}_\lambda)^{-1}(0)$ is $0$-deep in $C_0$ so that $\lambda_0 - s\pi(\tld{w}_h\tld{w}_\lambda)^{-1}(0)+s\pi\tld{w}^{-1}(0)+\pi(\nu'')$ is $m$-deep in $C_0$. 
Then $\lambda_0 - s\pi(\tld{w}_h\tld{w}_\lambda)^{-1}(0)+s\pi\tld{w}^{-1}(0)$ is in $C_0$ by Lemma \ref{lemma:packet}. 
\end{proof}

\begin{lemma}\label{lemma:domdirection}
Suppose that the center $Z$ of $G$ is connected. 
Let $\lambda_0 \in C_0$, $\lambda' \in X_1(T)$, $\tld{w}_\lambda \in \tld{W}_1$ with image $w_\lambda\in W$, $\tld{w}\in \tld{W}$, $\nu\in X(T)^+$, and $\nu' \in \Conv(\nu)$ such that 
\begin{enumerate}[label=(\arabic*)]
\item \label{item:small} $\langle \lambda_0+\eta,\alpha^\vee \rangle < p - h_\eta-\max_{v\in \ovl{A}_0} h_{\tld{w}_h\tld{w}_\lambda(v)}$ for all roots $\alpha$; 
\item \label{item:cover} $t_\nu\tld{w}_\lambda \cdot \lambda_0 \uparrow \lambda'+\pi(\nu')$; 
\item \label{item:dominant} $\tld{w}\cdot (\lambda_0 + \pi \tld{w}_h\tld{w}_\lambda(0)+\pi w_0w_\lambda \tld{w}^{-1}(0))+\eta \in X(T)^+$; and 
\item \label{item:jantzen} $\tld{w}\cdot (\lambda_0 + \pi \tld{w}_h\tld{w}_\lambda(0)+\pi w_0w_\lambda \tld{w}^{-1}(0)) \uparrow \tld{w}_h \cdot \lambda'$. 
\end{enumerate}
Then, letting $\tld{w}_{\lambda'} \in t_\nu\tld{w}_\lambda W_a\cap \tld{W}_1$ such that $\lambda' \in \tld{w}_{\lambda'}\cdot \ovl{C}_0$, we have 
\begin{itemize}
\item $\lambda'+\pi(\nu') \in \tld{w}_{\lambda'}W\cdot C_0$; 
\item $\lambda_0+\pi\tld{w}_h\tld{w}_\lambda(0)+\pi w_0w_\lambda\tld{w}^{-1}(0) \in C_0$; and
\item $t_\nu \tld{w}_\lambda \uparrow \tld{w}_{\lambda'}$.
\end{itemize} 
\end{lemma}
\begin{proof}
Let $w_\lambda\in W$ be the image of $\tld{w}_\lambda$. 
Then 
$$
\langle \lambda_0 - w_\lambda^{-1}\pi(\nu')+\eta,\alpha^\vee\rangle < p - h_\eta-\max_{v\in \ovl{A}_0} h_{\tld{w}_h\tld{w}_\lambda(v)}+h_\nu \leq p
$$ 
for all roots $\alpha$ by \ref{item:small}, \ref{item:cover}, and Lemma \ref{lemma:packet}, so that $\sigma_1\cdot (\lambda_0 - w_\lambda^{-1}\pi(\nu')) \in \ovl{C}_0$ for some $\sigma_1\in W$. 
Letting $\lambda'_0 \defeq \tld{w}_{\lambda'} ^{-1}\cdot \lambda' \in \ovl{C}_0$, Lemma \ref{lemma:LAP} with $\lambda_0$, $\mu_0$, $\tld{w}_\lambda$, and $\tld{w}_\mu$ taken to be $\sigma_1\cdot (\lambda_0 - w_\lambda^{-1}\pi(\nu'))$, $\lambda_0'$, $\pi(t_\nu\tld{w}_\lambda\sigma_1^{-1})$, and $\pi(\tld{w}_{\lambda'})$ implies that $\lambda'_0 = \sigma_1\cdot (\lambda_0 - w_\lambda^{-1}\pi(\nu'))$. 
Since 
\begin{equation}\label{eqn:uparrowword}
t_\nu\tld{w}_\lambda \sigma_1^{-1}\cdot \lambda_0' \uparrow \tld{w}_{\lambda'} \cdot \lambda_0', 
\end{equation}
we also have that $t_\nu\tld{w}_\lambda(0) \leq \tld{w}_{\lambda'}(0)$. 

Next, we claim that $h_{\tld{w}(0)} \leq h_\eta$. 
Let $w\in W$ be the image of $\tld{w}$ and suppose that $\alpha$ is a root so that $\langle \tld{w}(0),\alpha^\vee\rangle = h_{\tld{w}(0)}$. 
Then we claim that
\begin{align*}
(p-1)h_{\tld{w}(0)} - (p-1-h_\eta) &\leq \langle p\tld{w}(0),\alpha^\vee\rangle+\langle \pi w_0w_\lambda \tld{w}^{-1}(0),w^{-1}\alpha^\vee\rangle \\
&\quad + \langle \lambda_0 + \eta,w^{-1}\alpha^\vee\rangle + \langle \pi\tld{w}_h\tld{w}_\lambda(0),w^{-1}\alpha^\vee \rangle \\
&= \langle p\tld{w}(0)+w(\lambda_0 + \eta + \pi\tld{w}_h\tld{w}_\lambda(0)+\pi w_0w_\lambda \tld{w}^{-1}(0)),\alpha^\vee \rangle \\
&= \langle \tld{w}\cdot (\lambda_0 + \pi \tld{w}_h\tld{w}_\lambda(0)+\pi w_0w_\lambda \tld{w}^{-1}(0))+\eta,\alpha^\vee \rangle \\
&\leq \langle \tld{w}\cdot (\lambda_0 + \pi \tld{w}_h\tld{w}_\lambda(0)+\pi w_0w_\lambda \tld{w}^{-1}(0))+\eta,\alpha_0^\vee \rangle \\
&\leq \langle \tld{w}_h\cdot \lambda' +\eta,\alpha_0^\vee\rangle \\ 
&\leq (p-1)h_\eta
\end{align*}
where $\alpha_0$ is some highest root. Indeed, 
\begin{itemize}
\item the first inequality follows from \ref{item:small}; 
\item the second inequality follows from  \ref{item:dominant}; 
\item the third inequality follows from \ref{item:jantzen} and the dominance of $\alpha_0$; and  
\item the final inequality follows from the fact that $\lambda' \in X_1(T)$. 
\end{itemize}
Thus, $h_{\tld{w}(0)} \leq \frac{p-2}{p-1}h_\eta+1 < h_\eta+1$, and the claim follows. 

From the previous claim and \ref{item:small}, we have $\langle \lambda_0+\eta+\pi\tld{w}_h\tld{w}_\lambda(0)+\pi w_0w_\lambda\tld{w}^{-1}(0),\alpha^\vee\rangle < p$ for all roots $\alpha$. 
Then $\sigma_2\cdot (\lambda_0+\pi\tld{w}_h\tld{w}_\lambda(0)+\pi w_0w_\lambda\tld{w}^{-1}(0)) \in \ovl{C}_0$ for some $\sigma_2\in W$ and in particular $\tld{w}(0)\in X(T)^+$. 
\ref{item:jantzen} and Lemma \ref{lemma:LAP} with $\lambda_0$, $\mu_0$, $\tld{w}_\lambda$, and $\tld{w}_\mu$ taken to be $\sigma_2\cdot (\lambda_0+\pi\tld{w}_h\tld{w}_\lambda(0)+\pi w_0w_\lambda\tld{w}^{-1}(0))$, $\lambda_0'$, $\pi(\tld{w}\sigma_2^{-1})$, and $\pi(\tld{w}_h \tld{w}_{\lambda'})$ then imply that $\lambda'_0 = \sigma_2\cdot (\lambda_0+\pi\tld{w}_h\tld{w}_\lambda(0)+\pi w_0w_\lambda\tld{w}^{-1}(0))$ and $\tld{w}(0) \leq \tld{w}_h\tld{w}_{\lambda'}(0)$ (as $0$ lies in the closure of the facet determined by $\lambda'_0$). 
Putting things together, we have $\lambda_0+\pi\tld{w}_h\tld{w}_\lambda(0)+\pi w_0w_\lambda\tld{w}^{-1}(0) = \sigma_2^{-1}\sigma_1 \cdot (\lambda_0 - w_\lambda^{-1}\pi(\nu'))$, or equivalently that 
\begin{equation}\label{eqn:reflection}
\lambda_0+\pi\tld{w}_h\tld{w}_\lambda(0)+\pi w_0w_\lambda\tld{w}^{-1}(0) + \sigma_2^{-1}\sigma_1 w_\lambda^{-1}\pi(\nu') = \sigma_2^{-1}\sigma_1 \cdot \lambda_0. 
\end{equation}
Now
\begin{align}\label{eqn:rootorder}
\nonumber \tld{w}_h\tld{w}_\lambda(0)+ w_0w_\lambda\tld{w}^{-1}(0) + \pi^{-1}(\sigma_2^{-1}\sigma_1 w_\lambda^{-1})(\nu') &\geq \tld{w}_h\tld{w}_\lambda(0)+ w_0w_\lambda\tld{w}^{-1}(0) + w_0 \nu \\
&\geq w_0 \nu + \tld{w}_h\tld{w}_\lambda(0)- \tld{w}(0) \\
\nonumber &= \tld{w}_ht_\nu\tld{w}_\lambda(0)- \tld{w}(0) \\
\nonumber &\geq 0 
\end{align}
where 
\begin{itemize}
\item the second inequality uses that $\tld{w}(0) \geq -w_0w_\lambda\tld{w}^{-1}(0)$ since both sides are in the same $W$-orbit and $\tld{w}(0) \in X(T)^+$; and 
\item the last inequality uses the inequalities $t_\nu\tld{w}_\lambda(0) \leq \tld{w}_{\lambda'}(0)$ and $\tld{w}(0) \leq \tld{w}_h\tld{w}_{\lambda'}(0)$ proven in the first and third paragraphs, respectively. 
\end{itemize}
\eqref{eqn:reflection} and \eqref{eqn:rootorder} imply that $\lambda_0 \leq \sigma_2^{-1}\sigma_1\cdot \lambda_0$ so that $\sigma_1=\sigma_2$ and $w_\lambda^{-1}\pi(\nu') = \pi w_0\nu$. 
In particular, $\lambda_0 - w_\lambda^{-1}\pi(\nu') = \lambda_0-\pi w_0\nu \in C_0$ by \ref{item:small} and Lemma \ref{lemma:packet}\ref{lemma:packet:it:2} (which applies by hypothesis \ref{item:cover}). 
By definition of $\sigma_1$, we conclude that $\sigma_1 = 1$ and thus $\sigma_2 = 1$. 

We now deduce the three desired conclusions. 
First, we have $\lambda'_0 = \lambda_0 - w_\lambda^{-1}\pi(\nu') = \lambda_0+\pi\tld{w}_h\tld{w}_\lambda(0)+\pi w_0w_\lambda\tld{w}^{-1}(0) \in C_0$. 
In conjunction with \eqref{eqn:uparrowword}, this gives the inequality $t_\nu \tld{w}_\lambda \uparrow \tld{w}_{\lambda'}$. 

Finally, $\lambda'+\pi(\nu') = \tld{w}_{\lambda'}\cdot (\lambda_0-w_\lambda^{-1}\pi(\nu')+w_{\lambda'}^{-1}\pi(\nu'))$. 
Then $\lambda_0-w_\lambda^{-1}\pi(\nu')+w_{\lambda'}^{-1}\pi(\nu')\in W\cdot \ovl{C}_0$ by \ref{item:small} and Lemma \ref{lemma:packet} so that $\lambda'+\pi(\nu')$, which is linked to $\lambda_0\in C_0$, is in $\tld{w}_{\lambda'}W\cdot C_0$. 
\end{proof}

\begin{thm}\label{thm:JHfactor}
Suppose that the center $Z$ of $G$ is connected. 
Let $\tld{w}_\lambda \in \tld{W}_1$, $s \in W$, and $\lambda_0 \in C_0$. 
Suppose further that either 
\begin{enumerate}[label=(\arabic*)]
\item \label{item:alldirection} $p>(h_\eta+1)^2$ and that $\lambda_0$ is $\max_{v\in \ovl{A}_0} h_{\tld{w}_h\tld{w}_\lambda(v)}$-deep in $C_0$; or 
\item \label{item:domdirection} $\langle \lambda_0+\eta,\alpha^\vee \rangle < p - h_\eta-\max_{v\in \ovl{A}_0} h_{\tld{w}_h\tld{w}_\lambda(v)}$ for all roots $\alpha$ and $s = \pi(w_0w_\lambda)$. 
\end{enumerate}
Then $F(\tld{w}_\lambda \cdot \lambda_0)$ is a Jordan--H\"older factor of $\ovl{R}_s(\lambda_0+\eta-s\pi(\tld{w}_h\tld{w}_\lambda)^{-1}(0))$ with multiplicity one. 
\end{thm}
\begin{proof}
Let $\lambda \defeq \tld{w}_\lambda \cdot \lambda_0$ and $\mu\defeq\lambda_0+\eta-s\pi(\tld{w}_h\tld{w}_\lambda)^{-1}(0)$.
We apply Corollary \ref{cor:jantzendeep} with $\tld{w} \defeq \tld{w}_h\tld{w}_\lambda$ to get
\begin{align*}
\dim \Hom_\Gamma(Q_\lambda,\ovl{R}_s(\mu)) &= [\widehat{Z}(1,\lambda_0-p(\tld{w}_h\tld{w}_\lambda)^{-1}(0))+p\eta):\widehat{L}(1,\lambda)]_{G_1T} \\
&= [\widehat{Z}(1,w_\lambda(\lambda_0-p(\tld{w}_h\tld{w}_\lambda)^{-1}(0)+\eta) + (p-1)\eta) :\widehat{L}(1,\lambda)]_{G_1T} \\
&= [\widehat{Z}(1,\lambda) :\widehat{L}(1,\lambda)]_{G_1T}\\
&= 1 
\end{align*}
where $w_\lambda \in W$ is the image of $\tld{w}_\lambda$, the second equality follows from \cite[9.16(5)]{RAGS} and the fourth equality follows for instance by using that $[\widehat{Z}(1,\lambda) :\widehat{L}(1,\lambda)]_{G_1T}$ is nonzero and $\lambda$ appears with multiplicity one in both $\widehat{Z}(1,\lambda)|_T$ and $\widehat{L}(1,\lambda)|_T$. 
We claim that if $\lambda'\in X_1(T)$ and $[L(\lambda') \otimes L(\pi(\nu)):L(\lambda+p\nu)]_G\neq 0$ for some $\nu \in X(T)^+$ and $\Hom_\Gamma(Q_{\lambda'},\ovl{R}_s(\mu)) \neq 0$, then $\nu \in X^0(T)$. 
This will finish the proof: first note that $[\ovl{R}_s(\mu):F(\lambda')]_\Gamma\neq 0$ implies $\Hom_\Gamma(Q_{\lambda'},\ovl{R}_s(\mu))\neq 0$.
Thus the claim shows that the only term in the first sum in equation \eqref{eq:mults} of Corollary \ref{cor:jantzen} that contributes is the term where $\lambda'\equiv \lambda$ modulo $(p-\pi)X^0(T)$ and that both factors of this term are $1$.

We now prove the claim.
Suppose that $\lambda'\in X_1(T)$, $[L(\lambda') \otimes L(\pi(\nu)):L(\lambda+p\nu)]_G\neq 0$ for some $\nu \in X(T)^+$, and $\Hom_\Gamma(Q_{\lambda'},\ovl{R}_s(\mu))\neq 0$. 
Then Lemma \ref{lemma:packet} implies that there exists $\nu'\in \Conv(\nu)$ such that $\lambda+p\nu \uparrow \lambda'+\pi(\nu')$, and Theorem \ref{thm:jantzen} implies that there exists $\tld{w} \in \tld{W}$ such that $\tld{w}\cdot (\mu-\eta+s\pi\tld{w}^{-1}(0))+\eta \in X(T)^+$ and $\tld{w}\cdot (\mu-\eta+s\pi\tld{w}^{-1}(0)) \uparrow \tld{w}_h \cdot \lambda'$. 
Lemmas \ref{lemma:anydirection} and \ref{lemma:domdirection} imply conditions \ref{item:Weyl} and \ref{item:C0} of Lemma \ref{lemma:vertex}. 
Finally, Lemma \ref{lemma:vertex} implies that $\nu\in X^0(T)$. 
\end{proof}

\section{Applications to weight elimination}

Let $q$ be a power of $p$ and $K = W(\F_q)[p^{-1}]$. 
Let $\Q_p\subset E\subset \ovl{\Q}_p$ be a sufficiently large finite extension of $\Q_p$. 
Let $\cO$ be the ring of integers of $E$ and $\F$ the residue field. 
Assume that any homomorphism $K \ra \ovl{\Q}_p$ factors through $E$. 
We now take $G_0$ to be $\Res_{\F_q/\F_p} \GL_{n/\F_q}$. 
Let $^L G \defeq \prod_{K\ra E} \GL_{n/\Z} \rtimes \Gal(E/\Q_p)$. 
Let $G_{\Q_p} \defeq \Gal(\ovl{\Q}_p/\Q_p)$ and $G_K \defeq \Gal(\ovl{K}/K)$ with inertia subgroups $I_{\Q_p}$ and $I_K$, respectively. 
Restriction and projection gives a bijection between conjugacy classes of continuous $L$-homomorphisms ${}^L\rhobar: G_{\Q_p} \ra\, ^L G(\F)$ and conjugacy classes of continuous homomorphisms $\rhobar: G_K \ra \GL_n(\F)$. 

Let $X_{\mathrm{reg}}^*(T) \subset X_1(T)$ denote the subset of $\lambda \in X_1(T)$ such that $\langle \lambda,\alpha^\vee\rangle<p-1$ for all simple roots $\alpha$. 
We say that a simple $\F[\GL_n(\F_q)]$-module is \emph{regular} if its highest weight is in $X_{\mathrm{reg}}^*(T)$. 
The $p$-dot action of $\tld{w}_h$ defines a self-bijection $X_{\mathrm{reg}}^*(T) \ra X_{\mathrm{reg}}^*(T)$. 
Let $\cR$ be the corresponding self-bijection on the set isomorphism classes of regular simple $\F[\GL_n(\F_q)]$-modules, that is $\cR(F(\lambda)) = F(\tld{w}_h\cdot \lambda)$. 
For a conjugacy class of tame continuous homomorphisms $\rhobar: G_K \ra \GL_n(\F)$, let $W^?(\rhobar)$ be the set $W^?({}^L\rhobar|_{I_{\Q_p}})$ in \cite[Definition 9.2.5]{GHS} where ${}^L\rhobar$ is the $L$-parameter corresponding to $\rhobar$. 
Outside degenerate cases which are irrelevant in our context, the set $W^?(\rhobar)$ has the following concrete description \cite[Proposition 9.2.3]{GHS}: we can write ${}^L\rhobar|_{I_{\Qp}}$ (in possibly several ways) as an explicit representation $\tau(s,\mu)$ depending on $(s,\mu-\eta)\in W\times (C_0\cap X^*(T))$ \cite[Proposition 9.2.3]{GHS}, and the set $W^?(\rhobar)$ is $\cR(\JH(\ovl{R}_s(\mu)))$ (in this case we say that $\rhobar$ is $m$-generic if we can choose $\mu-\eta$ to be $m$-deep in $C_0$).

It is conjectured \cite{GHS} that the set $W^?(\rhobar)$ controls weights of mod $p$ automorphic forms for any globalization of $\rhobar$ e.g.~mod $p$ Langlands parameters contributing to spaces of mod $p$ algebraic modular forms on definite unitary groups as in \S \ref{intro:app}.

In any such context, establishing the upper bound given by $W^?(\rhobar)$ is referred to as ``weight elimination''.
In our previous work \cite{LLL}, we establish weight elimination in an axiomatic framework that applies to many global contexts (for instance the one in Theorem \ref{thm:WE:intro}) under the hypothesis that $\rhobar$ is $(6n-2)$-generic.

The method of \emph{loc.~cit}.~was to combine constraints from $p$-adic Hodge theory with generic decomposition patterns of Deligne--Lusztig representations.
Our new results on the latter allow us to improve our earlier axiomatic weight elimination results to the following
\begin{thm}
Let $\rhobar: G_K \ra \GL_n(\F)$ be a continuous homomorphism. Suppose that $W(\rhobar)$ is a set of isomorphism classes of simple $\F[\GL_n(\F_q)]$-modules, $w\in W$, and $\nu\in X^*(T)$ such that if $W(\rhobar) \cap \JH(\ovl{R}_w(\nu)) \neq \emptyset$ and either
\begin{itemize}
\item $\tau(w,\nu)$ is regular (i.e.~multiplicity free); or
\item $w = 1$;
\end{itemize}
then $\rhobar$ has a potentially semistable lift of type $(\eta,\tau(w,\nu))$.
Assume further that $\rhobar^{\semis}|_{I_{K}}$ is $(2n+1)$-generic.

Then $W(\rhobar) \subset W^?(\rhobar^{\semis})$. 
\end{thm}
\begin{proof}
We follow the general outline of \cite[\S 4.2]{LLL}.

By \cite[Lemma 5]{Enns}, if $\rhobar$ has a potentially semistable lift of type $(\eta,\tau(w,\nu))$, so does $\rhobar^{\semis}$. 
We then reduce to the case where $\rhobar$ is semisimple. 

Suppose that $\lambda \in X_1(T)$ with $F(\lambda) \in W(\rhobar)$. 
First, $\rhobar$ is $(2n+1)$-generic in the sense of \cite[Definition 2]{Enns} so that $\lambda$ is $p$-regular by \cite[Theorem 8]{Enns}. 
Then we can write $\lambda = \tld{w}_\lambda \cdot \lambda_0$ with $\lambda_0 \in C_0$ (and $\tld{w}_\lambda \in \tld{W}_1$). 

In order for a $(2n+1)$-generic Galois representation to exist at all, one must have $p > (2n+1)n$ by Remark \ref{rmk:depth}. 
Thus, we assume without loss of generality that $p>(2n+1)n$. 
Using this, after possibly replacing $\lambda_0$ by an element in $\Omega\cdot \lambda_0$, we can assume that $\langle \lambda_0 +\eta,\alpha^\vee\rangle < p - 2h_\eta$ for all $\alpha\in R$. 
We claim that $\lambda_0$ is $(h_{\tld{w}_h\tld{w}_\lambda(0)}+1)$-deep in $C_0$. 
Suppose otherwise. 
Then Theorem \ref{thm:JHfactor}\ref{item:domdirection} implies that $F(\lambda) \in \JH(\ovl{R}_w(\nu))$ with $w = \pi(w_0w_\lambda)$ and $\nu = \lambda_0+\eta-\pi(w_0w_\lambda(\tld{w}_h\tld{w}_\lambda)^{-1}(0))$. 
Since $\tld{w}_\lambda\in \tld{W}_1$, we have $0\leq \langle -w_0w_\lambda(\tld{w}_h\tld{w}_\lambda)^{-1}(0),\alpha^\vee\rangle \leq 1$ for all $\alpha\in \Delta$ and $\langle \lambda_0 +\eta-\pi(w_0w_\lambda(\tld{w}_h\tld{w}_\lambda)^{-1}(0)),\alpha^\vee\rangle < p - h_\eta$ for all $\alpha \in R$. 
We conclude that $\rhobar$ has a potentially semistable (and thus potentially crystalline) lift of type $(\eta,\tau)$ where $\tau$ is the (regular) inertial type $\tau(w,\nu)$ and $\nu-\eta$ is $1$-deep, but not $(h_{\tld{w}_h\tld{w}_\lambda(0)}+2)$-deep. 
We claim that $\rhobar|_{I_{K}} \cong \ovl{\tau}(s,\mu)$ for some $s\in W$ and $\mu\in X^*(T)$ such that 
\begin{equation}\label{eqn:adm}
(t_\nu w)^{-1}t_\mu s \in \Adm(\eta)\defeq \{\tld{w}\in \tld{W}\mid \tld{w}\leq t_{\sigma(\eta)}\textrm{ for some }\sigma\in W\}. 
\end{equation}
Given this claim, we deduce that $\mu-\eta$ is not $(2h_\eta+2)$-deep so that $\rhobar$ is not $(2h_\eta+3)$-generic. 
(Indeed, if $\mu-\eta$ is $(2h_\eta+2)$-deep, then \eqref{eqn:adm} implies that $\nu-\eta$ is $(h_\eta+2)$-deep and thus $(h_{\tld{w}_h\tld{w}_\lambda(0)}+2)$-deep.) 
This is a contradiction. 

We now prove the claim in the previous paragraph. 
There is a unique lowest alcove presentation $(s,\mu-\eta)$ of $\rhobar$ compatible in the sense of \cite[\S 2.4]{MLM} with the $1$-generic lowest alcove presentation $(w,\nu-\eta)$ of $\tau$ by \cite[Lemma 2.4.4]{MLM} (though it should be assumed in the cited lemma that the center $Z$ is connected). 
Moreover, $\mu-\eta$ is $2n$-deep by \cite[Proposition 2.2.15]{LLL}. 
Let $K'/K$ be a finite unramified extension so that the restriction $\rhobar'$ is a direct sum of characters and the base change inertial type $\tau'$ is principal series. 
Furthermore, $(w',\nu'-\eta')$ is a $1$-generic lowest alcove presentation for $\tau'$ where $w'_{j'} = w_{j'|_{K}}$ and $\nu'_{j'} = \nu_{j'|_K}$ for any embedding $j': K'\ra E$, and similarly $(s',\mu'-\eta')$ is a compatible $2n$-generic lowest alcove presentation of $\rhobar'$. 
\cite[Theorem 3.2.1]{LLL} implies that $\tau(s',\mu') \cong \tau(x',\xi')$ for some $\xi' \in (\Z^n)^{\Hom_{\Q_p}(K',E)}$ and $x' \in S_n^{\Hom_{\Q_p}(K',E)}$ such that $(t_{\nu'}w')^{-1}t_{\xi'} x' \in \Adm(\eta')$. 
In particular, $\xi'-\eta'$ is $(-n+2)$-deep in $C_0$. 
Since $\mu'-\eta'$ is $2n$-deep in $C_0'$, the centralizer of the semisimple element corresponding to $\theta_{s',\mu'}$ is a maximal torus by the proof of \cite[Lemma 10.1.10]{GHS}. 
Since the corresponding centralizer for $\theta_{x',\xi'}$ is conjugate, it is also a maximal torus. 
We conclude that both $(T_{s'},\theta_{s',\mu'})$ and $(T_{x'},\theta_{x',\xi'})$ are maximally split. 
By \cite[Proposition 9.2.1]{GHS}, $(T_{s'},\theta_{s',\mu'})$ and $(T_{x'},\theta_{x',\xi'})$ are $G'_0(\F_p)$-conjugate. 
By \cite[Lemma 4.2]{herzig-duke}, we have
\[
(x',\xi') = (\sigma s' \pi' \sigma^{-1} \pi'^{-1},\sigma(\mu^\prime) + (p-\sigma s' \pi' \sigma^{-1})\omega) 
\]
for some $\sigma \in S_n^{\Hom_{\Q_p}(K',E)}$ and $\omega \in (\Z^n)^{\Hom_{\Q_p}(K',E)}$. 
By Lemma \ref{lemma:presentation} taking $m=2n$, we have that $t_\omega\sigma \in \Omega'$. 
Since $\mu'-\xi' \in \Z R'$ and $\tld{W}'/W_a'$ is torsion-free, $t_\omega \sigma \in \Omega' \cap W_a' = 1$. 
Thus, $s' = x'$ and $\mu' = \xi'$ so that $(t_{\nu'} w')^{-1}t_{\mu'} s' \in \Adm(\eta')$ which implies that $(t_\nu w)^{-1}t_\mu s \in \Adm(\eta)$. 

Thus, $\lambda_0$ is $(h_{\tld{w}_h\tld{w}_\lambda(0)}+1)$-deep in $C_0$. 
Then $F(\lambda) \in \JH(\ovl{R}_w(\lambda_0+\eta-w\pi(\tld{w}_h\tld{w}_\lambda)^{-1}(0)))$ for all $w\in W$ by Theorem \ref{thm:JHfactor}\ref{item:alldirection} (recall that we are assuming that $p>(2n+1)n \geq n^2 = (h_\eta+1)^2$), which implies that $\rhobar$ has a potentially semistable (and thus potentially crystalline) lift of type $\tau_w \defeq \tau(w,\lambda_0+\eta-w\pi(\tld{w}_h\tld{w}_\lambda)^{-1}(0))$ for all $w\in W$. 
For all $w\in W$, $(w,\lambda_0-w\pi(\tld{w}_h\tld{w}_\lambda)^{-1}(0))$ is a $1$-generic lowest alcove presentation for $\tau_w$. 
The argument in the previous paragraph shows that $\rhobar|_{I_{K}} \cong \tau(s,\mu)$ for some $s\in W$ and $\mu\in X^*(T)$ with $(t_{\lambda_0+\eta-w\pi(\tld{w}_h\tld{w}_\lambda)^{-1}(0)}w)^{-1}t_\mu s \in \Adm(\eta)$ for all $w\in W$. 
Then the proof of \cite[Lemma 4.1.10]{LLL} implies that $F(\lambda) \in W^?(\rhobar)$. 
\end{proof}

\bibliography{Biblio}
\bibliographystyle{amsalpha}

\end{document}